\documentclass{amsart}
\usepackage{amssymb}
\usepackage{mathrsfs}
\usepackage{stmaryrd}
\usepackage{bbm}
\usepackage{oldgerm}
\usepackage[english]{babel}
\usepackage[T1]{fontenc}
\usepackage[latin1]{inputenc}
\usepackage[all]{xy}
\usepackage{mathabx}
\usepackage{graphicx}
\usepackage[colorlinks,linkcolor=red,anchorcolor=red,citecolor=blue]{hyperref}

\newtheorem{theorem}[subsection]{Theorem}
\newtheorem{proposition}[subsection]{Proposition}

\newtheorem{conjecture}[subsection]{Conjecture}

\theoremstyle{definition}

\newtheorem{remark}[subsection]{Remark}

\numberwithin{equation}{subsection}

\newcommand{\sO}{\mathscr O}
\newcommand{\sF}{\mathscr F}
\newcommand{\sK}{\mathscr K}
\newcommand{\sG}{\mathscr G}

\newcommand{\fm}{\mathfrak m}

\newcommand{\bQ}{\mathbb Q}
\newcommand{\bT}{\mathbb T}
\newcommand{\bA}{\mathbb A}

\newcommand{\cB}{\mathcal B}

\newcommand{\rR}{\mathrm R}
\newcommand{\Spec}{\mathrm{Spec}}
\newcommand{\sw}{\mathrm{sw}}
\newcommand{\SW}{\mathrm{SW}}
\newcommand{\DT}{\mathrm{DT}}
\newcommand{\GDT}{\mathrm{GDT}}
\newcommand{\LDT}{\mathrm{LDT}}
\newcommand{\GLDT}{\mathrm{GLDT}}
\newcommand{\dimtot}{\mathrm{dimtot}}
\newcommand{\rk}{\mathrm{rk}}
\newcommand{\Gr}{\mathrm{Gr}}
\newcommand{\Hom}{\mathrm{Hom}}

\newcommand{\wt}{\widetilde}
\newcommand{\ub}{\underline\beta}

\begin{document}

\title{Logarithmic ramifications of \'etale sheaves by restricting to curves}

\author{Haoyu Hu}

\address{Graduate School of Mathematical Sciences, the University of Tokyo, 3-8-1 Komaba Meguro-ku Tokyo 153-8914, Japan}

\email{huhaoyu1987@163.com, huhaoyu@ms.u-tokyo.ac.jp}



\begin{abstract}
In this article, we prove that the Swan conductor of an \'etale sheaf  on a smooth variety defined by Abbes and Saito's logarithmic ramification theory can be computed by its classical Swan conductors after restricting it to curves.  It extends the main result of \cite{barr} for rank $1$ sheaves. As an application, we give a logarithmic ramification version of generalizations of
Deligne and Laumon's lower semi-continuity property for Swan conductors of \'etale sheaves on relative curves to higher relative dimensions in a geometric situation.
\end{abstract}

\maketitle
\tableofcontents

\section{Introduction}
\subsection{}
In 1960s, the well-known Grothendieck-Ogg-Shafarevich formula was established to compute the Euler-Poincar\'e characteristic of an $\ell$-adic sheaf on a projective and smooth curve over an algebraically closed field of characteristic $p>0$ ($p\neq \ell$) (\cite{sga5}). The ramification theory of local fields contributes to this formula through an invariant called the {\it Swan conductor} (\cite{lr}). It measures the wild ramification of representations of the absolute Galois group of a local field.

In higher dimensions, the ramification phenomena involving imperfect residue extensions has been for a long time an obstacle to study ramifications of $\ell$-adic sheaves along divisors and to generalize the Grothendieck-Ogg-Shafarevich formula.  In 1970s, Deligne started a program to measure the ramification of an $\ell$-adic sheaf by restricting the sheaf to smooth curves and to generalize Grothendieck-Ogg-Shafarevich formula to higher dimensions. It was developed in \cite{lau2}. On the other hand, using $K$-theory, Kato initiated a study on the ramification of local fields with imperfect residue fields (\cite{kato1,kato2,kato3}). Later, Abbes and Saito introduced a more geometric approach that enables to study the ramification of an $\ell$-adic sheaf along a divisor (\cite{as1,as2,rc, saito cc, as3, wr}). Meanwhile, Kedlaya and Xiao established the ramification theory for $p$-adic differential modules, which is a parallel of Abbes and Saito's theory in the $p$-adic world (\cite{ked1,ked2,xl1,xl2}).  Naturally, we may ask whether there exists  a connection between Deligne's program and Abbes and Saito's approach.

\subsection{}\label{asramintro}
Let $K$ be a discrete valuation field,  $\sO_K$ its integer ring, $F$ the residue field of $\sO_K$, $\overline K$ a separable closure of $K$, $G_K$ the Galois group of $\overline K$ over $K$. We assume that the characteristic of $F$ is $p>0$ and that $\sO_K$ is henselian.
Abbes and Saito defined two decreasing filtrations $G^r_K$ and $G^r_{K,\log}$ ($r\in \bQ_{\geq 0}$) of $G_K$ by closed normal subgroups called the {\it ramification filtration} and the {\it logarithmic ramification filtration}, respectively (\cite[3.1,~3.2]{as1}). Let $M$ be a finitely generated $\Lambda$-module on which the wild inertia subgroup of $G_K$ acts, where $\Lambda$ denotes a finite field of characteristic $\ell\neq p$. Let $M=\bigoplus_{r\in \mathbb Q_{\geq 0}}M^{(r)}_{\log}$ be the logarithmic slope decomposition of $M$ and we define by
\begin{equation*}
\sw_KM=\sum_{r\in \mathbb Q_{\geq 0}}r\cdot\dim_{\Lambda}M^{(r)}_{\log}
\end{equation*}
the Swan conductor of $M$ (\ref{logmodsection}). When $F$ is a perfect field, we see that $G_K^{r+1}=G^{r}_{K,\log}$, that $G^{r}_{K,\log}$ coincides with its classical upper numbering filtration of $G_K$ and that $\sw_KM$ here coincides with the classical Swan conductor of $M$ (\cite[19.3]{lr}).

\subsection{}\label{swdivintro}
In the following of the introduction, let $\kappa$ be a perfect field of characteristic $p>0$. Let $Y$ be a smooth $\kappa$-scheme of finite type, $E$ a reduced Cartier divisor on $Y$, $V$ the complement of $E$ in $Y$, $g:V\rightarrow Y$ the canonical injection, $\{E_i\}_{i\in I}$ the set of irreducible components of $E$, $\xi_i$ the generic point of $E_i$, $Y_{(\xi_i)}$ the henselization of $Y$ at $\xi_i$, $\eta_i$ the generic point of $Y_{(\xi_i)}$, $K_i$ the fraction field of $Y_{(\xi_i)}$ and $\overline K_i$ a separable closure of $K_i$. Let $\sG$ be a locally constant and constructible \'etale sheaf of $\Lambda$-modules on $V$. For each $i\in I$, the restriction $\sG|_{\eta_i}$ corresponds to a finitely generated $\Lambda$-module with a continuous $\mathrm{Gal}(\overline K_i/K_i)$-action. We define the {\it Swan divisor} of $g_!\mathscr G$ on $Y$ and denote by $\SW_Y(g_!\sG)$ the Cartier divisor:
\begin{equation*}
\mathrm{SW}_Y(g_!\mathscr G)=\sum_{i\in I} \mathrm{sw}_{K_{i}} (\mathscr G|_{\eta_{i}})\cdot E_i.
\end{equation*}
We also define the {\it logarithmic total dimension divisor} of $g_!\mathscr G$ on $Y$ and denote by $\LDT_Y(g_!\sG)$ the Cartier divisor:
\begin{equation*}
\mathrm{LDT}_Y(g_!\mathscr G)=\sum_{i\in Y} \left(\mathrm{sw}_{K_{i}} (\mathscr G|_{\eta_{i}})+\dim_{\Lambda}\sG\right)\cdot E_i.
\end{equation*}

\subsection{}\label{ramdeligne}
Let $X$ be a {\it normal} $\kappa$-scheme of finite type, $D$ an effective Cartier divisor on $X$ such that the complement $U=X-D$ is smooth over $\Spec(\kappa)$, and $j:U\rightarrow X$ the canonical injection. We denote by $\mathcal C(X)$ the set of canonical morphisms $g:\widetilde C\rightarrow X$, where $\widetilde C$ is the normalization of a $1$-dimensional integral closed subscheme $C$ of $X$ such that $C\times_XD$ is an effective Cartier divisor on $C$. Let $\sF$ be an locally constant and constructible sheaf of $\Lambda$-modules on $U$. We say that  the ramification of $\sF$  along $D$ is bounded by a Cartier divisor $R$ (supported on $D$) {\it in the sense of Deligne} if, for any $g:\widetilde C\rightarrow X$ in $\mathcal C(X)$, we have (\cite[Definition 3.6]{ek})
\begin{equation*}
g^*R\geq \SW_{\widetilde C}(g^*(j_!\sF)).
\end{equation*}
Esnault and Kerz proved that each locally constant and constructible sheaf of $\Lambda$-modules on $U$ is bounded by a Cartier divisor in the sense of Deligne (\cite[Proposition 3.9]{ek}). They further predict that Abbes and Saito's logarithmic ramification theory gives the sharp lower bound:

\begin{conjecture}[{Esnault and Kerz, \cite{ek} and \cite[Conjecture A]{barr}}]\label{ekconja}
We take the notation and assumptions of \ref{ramdeligne}. Then, the following two conditions are equivalent:
\begin{itemize}
\item[1.] The ramification of $\sF$ along $D$ is bounded  by a Cartier divisor $R$ in the sense of Deligne,
\item[2.]
For any birational morphism $f:X'\rightarrow X$ such that $X'$ is smooth over $\Spec(\kappa)$, that $(X'\times_XD)_{\mathrm{red}}$ is a divisor on $X'$ with simple normal crossing and that $f^{-1}(U)=U$, we have (\ref{bigger})
\begin{equation*}
f^*R\geq \SW_{X'}(f^*(j_!\sF)).
\end{equation*}
\end{itemize}
\end{conjecture}

\begin{conjecture}[{Esnault and Kerz, \cite{ek} and \cite[Conjecture B]{barr}}]\label{ekconjb}
We take the notation and assumptions of \ref{ramdeligne} and we assume that the effective Cartier divisor $D$ is integral. Let $\mathcal{PC}(X)$ be the set of pairs $(g:\wt C\rightarrow X,\, x\in \wt C)$ where $g:\wt C\rightarrow X$ is an element of $\mathcal C(X)$ and $x$ is a closed point of  $g^*D=\wt C\times_XD$. Then, we have
\begin{equation}\label{formulaek2}
\sup_{\mathcal{PC}(X)}\frac{\sw_x(g^*j_!\sF)}{m_x(g^*D)}=\sw_D(j_!\sF),
\end{equation}
where $m_x(g^*D)$ denotes the multiplicity of $g^*D$ at $x\in \wt C$ and $\sw_D(j_!\sF)$ the coefficient of the Swan divisor of $j_!\sF$ after replacing $X$ by a smooth and open neighborhood of the generic point of $D$~(\ref{swdivintro}).
\end{conjecture}

The main result of this article is the following:
\begin{theorem}[{cf. Theorem \ref{thmconjb} and Theorem \ref{thmconja}}]\label{mainresultintro}
Conjecture \ref{ekconja} holds under the condition that $X$ is smooth over $\Spec(\kappa)$ and $D$ is a divisor with simple normal crossing.  Conjecture \ref{ekconjb} holds under the condition that $X$ is smooth over $\Spec(\kappa)$ and $D$ is irreducible and smooth over $\Spec(\kappa)$.
\end{theorem}

\subsection{}
In \cite{barr}, Barrientos obtained Theorem \ref{mainresultintro} assuming that $\sF$ has rank $1$. His proof relies on the equivalence of Abbes and Saito's logarithmic ramification theory and Brylinski and Kato's approach for characters (\cite{mla,bry,katochar}), and on an explicit computation of Witt vectors which emerges in {\it loc. cit.}.
In this article, to solve the problem for sheaves of any rank, we follow an idea of Saito which is different from Barrientos's method. The new idea roughly says that, after replacing the local field $K$ by a tamely ramified extension $K'$ of a sufficient large ramification index, the Swan conductor $\sw_{K'}(M)$ of a Galois module $M$ is {\it asymptotic} to the total dimension of $M$ with respect to $K'$ defined by the ramification filtration (\ref{subslopedecom}). Geometrically,  it says that Swan divisors are asymptotic to the total dimension divisors \eqref{dtdiv} after taking covers of $X$ tamely ramified along $D$. Then, we can obtain several inequalities that compare the pull-back of the Swan divisor of an \'etale sheaf and the Swan divisor of the pull-back of the sheaf (Proposition \ref{conj2strong} and Theorem \ref{proofconj1}), using analogous results for total dimension divisors of sheaves in \cite{HY,wr} (cf. Proposition \ref{hymain} and Proposition \ref{saitomain}). These inequalities immediately deduce main results of this article (Theorem \ref{thmconjb} and Theorem \ref{thmconja}).
Saito's idea also appeared in the proof of Bertini theorem due to Achinger (\cite[Theorem 1.3.1]{ach}).

The main result Theorem \ref{mainresultintro} has two applications. One is to obtain a geometric criterion for a Galois module having the same Swan conductor and total dimension (cf. Proposition \ref{sw=dimtot}). Another one is to give a logarithmic version of generalizations of Deligne and Laumon's semi-continuity property for Swan conductors of $\ell$-adic sheaves on relative curves to higher relative dimensions in geometric situations (\cite{lau}). We state the latter in the following.

\subsection{}\label{smfibintro}
Let $S$ be a separated $\kappa$-scheme of finite type, $g:X\rightarrow S$ a separated and smooth morphism of finite type and $D$ an effective Cartier divisor on $X$ relative to $S$, $U$ the complement of $D$ in $X$ and $j:U\rightarrow X$ the canonical injection. Let $\{S_\alpha\}_{\alpha\in J}$ be the set of irreducible component of $S$, $\bar\eta_{\alpha}$ an algebraic geometric generic point of $S_{\alpha}$, $X_{\alpha}=X\times_SS_{\alpha}$, $D_{\alpha}=D\times_SS_{\alpha}$ and $g_{\alpha}:X_{\alpha}\rightarrow S_{\alpha}$ the base change of $g:X\rightarrow S$ by the injection $S_{\alpha}\rightarrow S$. We assume that $D_\alpha$ is a sum of irreducible effective Cartier divisors $D_{\alpha i}$ $(i\in I_{\alpha})$ such that the restriction $g|_{D_{\alpha i}}:D_{\alpha i}\rightarrow S_{\alpha}$ is smooth. For each algebraic geometric point $\bar s\rightarrow S$, we denote by $\iota_{\bar s}:X_{\bar s}\rightarrow X$ the base change of
$\bar s\rightarrow S$ by $g:X\rightarrow S$.
Let $\sF$ be a locally constant and constructible sheaf of $\Lambda$-modules on $U$. For each $\alpha\in J$, we define the {\it generic logarithmic total dimension divisor} of $j_!\sF|_{X_{\alpha}}$ and denote by $\GLDT_{g_{\alpha}}(j_!\sF|_{X_{\alpha}})$ the unique Cartier divisor on $X_{\alpha}$ supported on $D_{\alpha}$ such that
\begin{equation*}
\iota_{\bar\eta_{\alpha}}^*(\GLDT_{g_{\alpha}}(j_!\sF|_{X_{\alpha}}))=\LDT_{X_{\bar\eta_{\alpha}}}(j_!\sF|_{X_{\bar\eta_{\alpha}}}).
\end{equation*}

\begin{theorem}\label{logsemicontintro}
We take the notation and assumptions of \ref{smfibintro}. For each algebraic geometric point $\bar s\rightarrow S$, we denote by $J_{\bar s}$ the subset of $J$ such that $\bar s$ maps to $S_{\alpha}$ for $\alpha\in J_{\bar s}$. Then, there exists a Zariski open dense subset $V$ of $S$ such that
\begin{itemize}
\item[(1)]
For any algebraic geometric point $\bar s\rightarrow V$ and any $\alpha\in J_{\bar s}$, we have
\begin{equation*}
\iota^*_{\bar s}(\GLDT_{g_{\alpha}}(j_!\sF|_{X_{\alpha}}))=\LDT_{X_{\bar s}}(j_!\sF|_{X_{\bar s}});
\end{equation*}
\item[(2)]
For any algebraic geometric point $\bar t\in S-V$ and any $\alpha \in J_{\bar t}$, we have
\begin{equation*}
\iota^*_{\bar t}(\GLDT_{g_{\alpha}}(j_!\sF|_{X_{\alpha}}))\geq\LDT_{X_{\bar t}}(j_!\sF|_{X_{\bar t}}).
\end{equation*}
\end{itemize}
\end{theorem}

\subsection{}
In \cite[Theorem 4.3]{HY} (cf. Theorem \ref{nonlogsemicont}), we obtained an analogous semi-continuity property for total dimension divisors, which is related to Abbes and Saito's ramification filtration. Here, we follow a similar strategy in {\it loc. cit.} to prove Theorem \ref{logsemicontintro}. Roughly speaking,  using inequalities that compare the pull-back of the Swan divisor of an \'etale sheaf and the Swan divisor of the pull-back of the sheaf (Proposition \ref{conj2strong} and Theorem \ref{proofconj1}), we reduce the theorem above to a relative curve situation. Then, we use the Deligne and Laumon's semi-continuity property \cite[2.1.1]{lau} to finish the proof.

\subsection{}
Besides our main result Theorem \ref{mainresultintro}, there have been several works on the same topic. Matsuda initiated a study of comparing total dimensions of rank $1$ sheaves after restricting to curves (\cite{mat}, cf. \cite[Proposition 2.7]{ks}). By studying the characteristic cycles of sheaves under certain ramification conditions, Saito gives a geometric criterion for the commutativity between total dimension divisors and restricting to curves (\cite{wr}, cf. Proposition \ref{saitomain}). This result also implies that the total dimension of an \'etale sheaf on smooth varieties can be computed by restricting the sheaf to curves. Later, an inequality that compares total dimension divisors of \'etale sheaves after arbitrary pull-back was found in \cite{HY} (cf. Proposition \ref{hymain}). In the logarithmic situation, Zhukov stated a conjectural ramification invariant of sheaves on smooth surfaces following Deligne's program (\cite[2.5.3]{zhu}). As mentioned already, Esnault and Kerz's conjectures above provide a concrete connection between Deligne's program and Abbes and Saito's logarithmic ramification theory. Barrientos answered their conjectures for rank $1$ sheaves on smooth varieties (\cite{barr}). We hope that Esnault and Kerz's conjectures can be wholly answered in future.

There have been several generalizations of Deligne and Laumon's semi-continuity property \cite[2.1.1]{lau}. Saito proved a semi-continuity property for total dimensions of stalks of vanishing cycles complexes (\cite{cc}). The lower semi-continuity property for total dimension divisors of $\ell$-adic sheaves on a smooth fibration is proved in \cite{HY}. The semi-continuity property for singular supports and characteristic cycles of $\ell$-adic sheaves is obtained in \cite{HY2}.

It is well known that algebraic $D$-modules are analogues of $\ell$-adic sheaves.  A local invariant of algebraic $D$-modules on smooth complex curves called the {\it irregularity} is an analogue of Swan conductors. In \cite{andre}, Andr\'e studied ramification of $D$-modules on higher dimensional complex varieties in terms of restricting to curves and proved a semi-continuity property for irregularities of $D$-modules on relative curves. We expect that semi-continuity properties in \cite{HY,HY2,cc}  have analogues in the theory of algebraic $D$-modules.

\subsection{}
This article is arranged as follows. To make it self-contained, we briefly recall Abbes and Saito's ramification theory and the theory of the singular support and the characteristic cycle in \S 3 and \S4. In \S5, we review several useful properties for the total dimension divisor and the Swan divisor depending on results in \S 3 and \S 4. The main result of this article will be proved in \S 6. In \S 7, we discuss a geometric criterion for a Galois module having the same Swan conductor and total dimension. In \S8, we prove a semi-continuity property for Swan divisors of $\ell$-adic sheaves on smooth fibrations.

\subsection*{Acknowledgement}
The author would like to express his sincere gratitude to T. Saito for sharing ideas, which are crucial to this work. The author is grateful to A. Abbes for pointing out the question on the semi-continuity for Swan divisors during a discussion on \cite{HY}. The author is supported by JSPS postdoctoral fellowship and Kakenhi Grant-in-Aid 15F15727 during his stay at the University of Tokyo.

\section{Notation}
\subsection{}
In this article, let $k$ be a field of characteristic $p>0$. We fix a prime number $\ell$ which is different from $p$ and a finite field $\Lambda$ of characteristic $\ell$. All $k$-schemes are assumed to be separated and of finite type over $\Spec(k)$ and all morphisms between $k$-schemes are assumed to be $k$-morphisms. All sheaves of $\Lambda$-modules on $k$-schemes are assumed to be \'etale sheaves.

\subsection{}\label{bigger}
Let $X$ be a Noetherian scheme and $D$ and $E$ two Cartier divisors on $X$. We write $D\geq E$ and say that $D$ is {\it bigger} than $E$ if $D-E$ is effective.

\subsection{}
 Let $x$ be a closed point of a $k$-scheme $X$. For any irreducible closed subscheme $Z$ of $X$ containing $x$, we denote by $m_x(Z)$ the multiplicity of $Z$ at $x$ (\cite[4.3]{fulton}).

\subsection{}
Let $X$ be a smooth $k$-scheme. We denote by $\bT X$ the tangent bundle of $X$ and by $\bT^*X$ the cotangent bundle of $X$. A {\it closed conical subset} of $\bT^*X$ denotes a reduced closed subscheme of $\bT^*X$ invariant under the canonical $\mathbb G_m$-action on $\bT^*X$. Let $Z$ be a  smooth $k$-scheme and $i:Z\rightarrow X$ a closed immersion. We denote by $\bT^*_ZX$ the conormal bundle of $Z$ in $X$. For any point $x$ of $X$, we put $\bT_xX=\bT X\times_Xx$ and $\bT^*_xX=\bT^*X\times_Xx$.

\subsection{}
Let $f:X\rightarrow S$ be a morphism of schemes, $s$ a point of $S$ and $\bar s\rightarrow S$ a geometric point above $s$. We denote by $X_s$ (resp. $X_{\bar s}$) the fiber $X\times_Ss$ (resp. $X\times_S\bar s$). Assume that $f:X\rightarrow S$ is flat and of finite presentation. Let $D$ be a Cartier divisor on $X$ relative to $S$ (\cite[IV, 21.15.2]{EGA4}). Let $\pi:S'\rightarrow S$ be a morphism of $k$-schemes, $X'=X\times_SS'$ and $\pi':X'\rightarrow X$ the base change of $\pi:S'\rightarrow S$. We denote by $\pi'^*D$ the pull-back of $D$, which is a Cartier divisor on $X'$ relative to $S'$  \cite[IV, 21.15.9]{EGA4}. When $S'$ is $s$ or $\bar s$, we simply denote by $D_s$ (resp. $D_{\bar s}$) the Cartier divisor $D\times_Ss$ on $X_s$ (resp. $D\times_S\bar s$ on $X_{\bar s}$). An effective Cartier divisor $E$ on $X$ relative to $S$ is identical to a closed immersion $i:E\rightarrow X$ transversally regular relative to $S$ and of codimension $1$ (\cite[IV,19.2.2 and 21.15.3.3]{EGA4}). The fiber $E_s$ (resp. $E_{\bar s}$) is an effective Cartier divisor on $X_s$ (resp. $X_{\bar s}$).

\section{Ramification theory of Abbes and Saito}
\subsection{}
In this section, $K$ denotes a discrete valuation field, $\sO_K$ its integer ring, $\fm_K$ the maximal ideal of $\sO_K$, $F$ the residue field of $\sO_K$, $\overline K$ a separable closure of $K$, $G_K$ the Galois group of $\overline K$ over $K$ and $\mathrm{ord}:\overline K\rightarrow \mathbb Q\bigcup\{\infty\}$ a valuation normalized by $\mathrm{ord}(K^{\times})=\mathbb Z$. We assume that the characteristic of $F$ is $p>0$ and that $\sO_K$ is henselian.
Abbes and Saito defined two decreasing filtrations $G^r_K$ and $G^r_{K,\log}$ ($r\in \bQ_{\geq 0}$) of $G_K$ by closed normal subgroups called the ramification filtration and the logarithmic ramification filtration, respectively (\cite[3.1,~3.2]{as1}).

\subsection{}\label{nonlogfil}
For any $r\in \bQ_{\geq 1}$, we put
\begin{equation*}
G^{r+}_K=\overline{\bigcup_{s\in \bQ_{>r}}G^s_K}.
\end{equation*}
We denote by $G_K^0$ the group $G_K$. For any $0<r\leq 1$, the subgroup $G^r_K$ is the inertia subgroup $I_K$ of $G_K$ and $G^{1+}_K$ is the wild inertia subgroup $P_K$ of $G_K$, i.e., the Sylow subgroup of $I_K$ \cite[3.7]{as1}. If $K$ has characteristic $p$, then, for any $r\in\bQ_{>1}$, the graded piece $\Gr^rG_K=G^r_K/G^{r+}_K$ is abelian, killed by $p$ and contained in the center of $P_K/G^{r+}_{K}$ (\cite[2.15]{as2}, \cite[Corollary 2.28]{wr} and \cite{xl1}).

\begin{proposition}[{\cite[Proposition 3.7 (2)]{as1}}]\label{nonlogext}
Let $K'$ be a finite separable extension of $K$ contained in $\overline K$ of ramification index $e$. We denote by $G_{K'}$ the Galois group of $\overline K$ over $K'$ and $G_{K'}^r$ ($r\in \bQ_{\geq 0}$) the ramification filtration of $G_{K'}$. Then, for any $r\in \bQ_{\geq 1}$, we have $G^{er}_{K'}\subseteq G^r_K$. If $K'$ is unramified over $K$, the inclusion is an equality.
\end{proposition}

\subsection{}\label{logfil}
For any $r\in \bQ_{\geq 0}$, we put
\begin{equation*}
G^{r+}_{K,\log}=\overline{\bigcup_{s\in \bQ_{>r}}G^s_{K,\log}}.
\end{equation*}
The subgroup $G_{K,\log}^0$ is the inertia subgroup $I_K$ of $G_K$ and $G^{1+}_{K,\log}$ is the wild inertia subgroup $P_K$ of $G_K$.
For any $r\in \bQ_{>0}$, the graded piece $\Gr_{\log}^rG_K=G^r_{K,\log}/G^{r+}_{K,\log}$ is abelian, killed by $p$ and contained in the center of $P_K/G^{r+}_{K,\log}$ (\cite[5.12]{as2}, \cite[Theorem 1.24]{saito cc}, \cite{as3}, \cite{xl1} and \cite{xl2}).

\begin{proposition}[{\cite[Proposition 3.15 (3)]{as1}}]\label{logext}
Let $K'$ be a finite separable extension of $K$ contained in $\overline K$ of ramification index $e$. We denote by $G_{K'}$ the Galois group of $\overline K$ over $K'$ and $G_{K',\log}^r$ ($r\in \bQ_{\geq 0}$) the ramification filtration of $G_{K'}$. Then, for any $r\in \bQ_{\geq 0}$, we have $G^{er}_{K',\log}\subseteq G^r_{K,\log}$. If $K'$ is tamely ramified over $K$, the inclusion is an equality.
\end{proposition}

\begin{proposition}[{\cite[Proposition 3.15 (1)(4)]{as1}}]\label{lnlcomp}
For any $r\in \bQ_{\geq 0}$, we have $G_K^{r+1}\subseteq G_{K,\log}^r\subseteq G_K^{r}$.  If the residue field $F$ is perfect, then we have $G^{r+1}_K=G^r_{K,\log}$ for any $r\in \bQ_{\geq 0}$ and the logarithmic ramification filtration of $G_K$ coincides with its classical upper numbering filtration.
\end{proposition}

\subsection{}\label{subslopedecom}
Let $M$ be a finitely generated $\Lambda$-module with a continuous $P_K$-action. The module $M$ has a decomposition (\cite[1.1]{K})
\begin{equation}\label{slopedecom}
M=\bigoplus\limits_{r\geq 1}M^{(r)}
\end{equation}
into $P_K$-stable submodules $M^{(r)}$, such that $M^{(1)}=M^{P_K}$ and, for every $r>1$,
\begin{equation*}
(M^{(r)})^{G^r_K}=0\ \ \ {\rm and}\ \ \  (M^{(r)})^{G^{r+}_K}=M^{(r)}.
\end{equation*}
The decomposition \eqref{slopedecom} is called the \emph{slope decomposition} of $M$. The values $r\geq 1$ for which  $M^{(r)}\neq 0$ are called the {\it slopes} of $M$. We say that $M$ is {\it isoclinic} if it has only one slope.

Assume that $M$ is isoclinic of slope $r>1$, that $K$ has characteristic $p$ and that $\Lambda$ contains a primitive $p$-th roots of unit. Then, the graded piece $\Gr^r G_K$ is abelian and killed by $p$ and $M$ has a faithful $\Gr^r G_K$-action. The module $M$ has a unique direct sum decomposition (cf. the proof of \cite[Lemma 6.7]{rc})
\begin{equation}\label{ccdecomp}
M=\bigoplus_{\chi\in X(r)} M_{\chi},
\end{equation}
into $P_K$-stable submodules $M_{\chi}$, where $X(r)$ denotes the set of isomorphism classes of non-trivial finite characters $\chi:\Gr^r G_K\rightarrow \Lambda^{\times}$ and $M_{\chi}$ is a direct sum of finitely many copies of $\chi$. The decomposition \eqref{ccdecomp} is called the {\it center character decomposition} of $M$. The characters $\chi\in X(r)$ for which $M_{\chi}\neq 0$ are called the {\it center characters} of $M$.

\subsection{}\label{logmodsection}
Let $M$ be a finitely generated $\Lambda$-module with a continuous $P_K$-action. The module $M$ has a
 decomposition (\cite[1.1]{K})
\begin{equation}\label{logslopedecom}
M=\bigoplus\limits_{r\geq 0}M^{(r)}_{\log}
\end{equation}
into $P_K$-stable submodules $M^{(r)}_{\log}$, such that $M^{(0)}_{\log}=M^{P_K}$ and, for every $r>0$,
\begin{equation*}
(M^{(r)}_{\log})^{G^r_{K,\log}}=0\ \ \ {\rm and}\ \ \  (M^{(r)}_{\log})^{G^{r+}_{K,\log}}=M^{(r)}_{\log}.
\end{equation*}
The decomposition \eqref{slopedecom} is called the \emph{logarithmic slope decomposition} of $M$. The values $r\geq 0$ for which  $M^{(r)}_{\log}\neq 0$ are called the {\it logarithmic slopes} of $M$.  The values $r\geq 0$ for which  $M^{(r)}\neq 0$ are called the {\it slopes} of $M$. We say that $M$ is {\it logarithmically isoclinic} if it has only one logarithmic slope.

Assume that $M$ is logarithmically isoclinic of slope $r>0$ and that $\Lambda$ contains a primitive $p$-th roots of unit.  Hence, the graded piece $\Gr_{\log}^rG_K$ is abelian and killed by $p$ and $M$ has a  faithful $\Gr^r_{\log} G_K$-action. The module $M$ has a unique direct sum decomposition (\cite[Lemma 6.7]{rc})
\begin{equation}\label{lccdecomp}
M=\bigoplus_{\chi\in Y(r)} M_{\log,\chi},
\end{equation}
into $P_K$-stable submodules $M_{\log,\chi}$, where $Y(r)$ denotes the set of isomorphism classes of non-trivial finite characters $\chi:\Gr_{\log}^r G_K\rightarrow \Lambda^{\times}$ and $M_{\log,\chi}$ is a direct sum of finitely many copies of $\chi$. The decomposition \eqref{lccdecomp} is called the {\it logarithmic center character decomposition} of $M$.
The characters $\chi\in Y(r)$ for which $M_{\log,\chi}\neq 0$ are called the {\it logarithmic center characters} of $M$.

\subsection{}\label{swdimtot}
Let $M$ be a finitely generated $\Lambda$-module with a continuous $P_K$-action. The \emph{total dimension} of $M$ is defined by
\begin{equation}\label{dtofmod}
\mathrm{dimtot}_K M=\sum_{r\geq 1}r\cdot \mathrm{dim}_{\Lambda}M^{(r)}.
\end{equation}
The \emph{Swan conductor} of $M$ is defined by
\begin{equation}\label{swofmod}
\mathrm{sw}_K M=\sum_{r\geq 0}r\cdot \mathrm{dim}_{\Lambda}M^{(r)}_{\log}.
\end{equation}
By Proposition \ref{lnlcomp}, we have
\begin{equation}\label{swdimtotcomp}
\sw_KM\leq\dimtot_KM\leq\sw_K M+\dim_\Lambda M.
\end{equation}
By {\it loc. cit.}, we see that
\begin{equation*}
\mathrm{sw}_KM=\mathrm{dimtot}_K M
\end{equation*}
if and only if
the slope decomposition and the logarithmic slope decomposition of $M$ are the same.
If the residue field $F$ is perfect, we have
\begin{equation}\label{swdimtotcurve}
\mathrm{dimtot}_K M=\mathrm{sw}_KM +\dim_{\Lambda }M
\end{equation}
and $\sw_K M$ equals the classical Swan conductor of $M$ (\cite[19.3]{lr}).

\begin{proposition}\label{unramtame}
Let $M$ be a finitely generated $\Lambda$-module with a continuous $P_K$-action and $K'$ a finite separable extension of $K$ contained in $\overline K$ of ramification index $e$. Then, we have
\begin{equation*}
\sw_{K'}M\leq e\cdot\sw_KM\ \ \ \textrm{and}\ \ \ \dimtot_{K'}M\leq e\cdot\dimtot_KM.
\end{equation*}
If $K'$ is tamely ramified over $K$, we have
\begin{equation}\label{tamesw}
\sw_{K'}M= e\cdot\sw_KM.
\end{equation}
If $K'$ is unramified over $K$, we have
\begin{equation}\label{unramdimtot}
\dimtot_{K'}M= \dimtot_KM.
\end{equation}
\end{proposition}
It is deduced from Proposition \ref{nonlogext} and Proposition \ref{logext}.

\begin{theorem}[Hasse-Arf Theorem, {\cite[3.4.3]{xl1} and  \cite[3.3.5, 3.5.14]{xl2}}]\label{xlmain}
Let $M$ be a finitely generated $\Lambda$-module with a continuous $G_K$-action. Then,
\begin{itemize}
\item[(i)] The total dimension $\dimtot_KM$ is a non-negative integer.
\item[(ii)] When $p=2$ and $K$ has characteristic $0$, we have $2\cdot\sw_KM\in \mathbb Z_{\geq 0}$. In other cases, the Swan conductor $\sw_KM$ is a non-negative integer.
\end{itemize}
\end{theorem}
Using Serre's {\it cde triangle} (\cite[Theorem 33]{lr}),  the theorem is reduced to the case where $M$ is a finite dimensional virtual $\overline{\mathbb Q}_{\ell}$-modules with a $G_K$-action that factors through a finite quotient group. By non-canonical isomorphisms $\overline{\mathbb Q}_{\ell}\cong\mathbb C\cong \overline{\mathbb Q}_{p}$, we may consider $M$ as a virtual $p$-adic representation of $G_K$ with finite monodromy. Hence, the theorem is deduced by Hasse-Arf theorems due to Xiao ({\cite[3.4.3]{xl1} and \cite[3.3.5, 3.5.14]{xl2}).

\subsection{}
In the rest of this section, we assume that $K$ has characteristic $p$ and that $F$ is of finite type over a perfect field. We denote by $\sO_{\overline K}$ the integral closure of $\sO_K$ in $\overline K$ and by $\overline F$ the residue field of $\sO_{\overline K}$. For any rational number $r$, we put
\begin{align*}
\fm^r_{\overline K}=\big\{x\in \overline K^{\times}\,;\,\mathrm{ord}(x)\geq r\big\}\ \ \ \textrm{and}\ \ \
\fm^{r+}_{\overline K}=\big\{x\in \overline K^{\times}\,;\,\mathrm{ord}(x)> r\big\}.
\end{align*}
Notice that the quotient $\fm^{r}_{\overline K}/\fm^{r+}_{\overline K}$ is a $1$-dimensional $\overline F$-vector space.
For any rational number $r>1$, there exists an injective homomorphism, call the {\it characteristic form}  (\cite[Corollary 2.28]{wr}),
\begin{equation}\label{charform}
\mathrm{char}:\Hom_{\mathbb F_p}(\Gr^rG_K,\mathbb F_p)\rightarrow \Hom_{\overline F}(\fm^{r}_{\overline K}/\fm^{r+}_{\overline K}, \Omega^1_{\sO_K}\otimes_{\sO_K}\overline F).
\end{equation}
Let $\Omega^1_{F}(\log)$ be the finitely dimensional $F$-vector space
\begin{equation*}
\Omega^1_{F}(\log)=(\Omega^1_{F}\oplus(F\otimes_{\mathbb Z}K^{\times}))/(d\bar a-\bar a\otimes a\,;\,a\in\sO_K^{\times}),
\end{equation*}
where $\bar a$ denotes the residue class of an element $a\in \sO_K$.
We have an exact sequence
\begin{equation*}
0\rightarrow \Omega^1_F\xrightarrow{\iota}\Omega^1_F(\log)\xrightarrow{\mathrm{res}}F\rightarrow 0,
\end{equation*}
where $\iota(a)=(a,0)$ for $a\in\Omega^1_F$ and $\mathrm{res}((0,b\otimes c))=b\cdot \mathrm{ord}(c)$ for $b\in F$ and $c\in K^{\times}$.
For any rational number $r>0$, there exists an injective homomorphism, call the {\it refined Swan conductor} (\cite[Corollary 2.28]{wr}),
\begin{equation}
\mathrm{rsw}:\Hom_{\mathbb F_p}(\Gr^r_{\log}G_K,\mathbb F_p)\rightarrow \Hom_{\overline F}(\fm^{r}_{\overline K}/\fm^{r+}_{\overline K},\Omega^1_F(\log)\otimes_{F}\overline F).
\end{equation}

\begin{proposition}[{\cite[Lemma 5.13]{as2}}]\label{diaglnl}
For each $r\geq 1$, The canonical inclusion $G^r_{K,\log}\subseteq G_K^r$ induces a map: $\theta_r:\Gr^r_{\log}G_K\rightarrow\Gr^rG_K$.
Let $\phi:\Omega^1_{\sO_K}\otimes_{\sO_K}F\rightarrow \Omega^1_F(\log)$ be the composition of the canonical projection $\pi:\Omega^1_{\sO_K}\otimes_{\sO_K}F\rightarrow \Omega^1_F$ and the injection $\iota:\Omega^1_F\rightarrow\Omega^1_F(\log)$.
Then, for each $r>1$, we have the following commutative diagram
\begin{equation}
\xymatrix{\relax
\Hom_{\mathbb F_p}(\Gr^rG_K,\mathbb F_p)\ar[r]^-(0.5){\mathrm{char}}\ar[d]_{\theta^\vee_r}&\Hom_{\overline F}(\fm^{r}_{\overline K}/\fm^{r+}_{\overline K}, \Omega^1_{\sO_K}\otimes_{\sO_K}\overline F)\ar[d]^{\Hom(\mathrm{id},\phi)}\\
\Hom_{\mathbb F_p}(\Gr^r_{\log}G_K,\mathbb F_p)\ar[r]_-(0.5){\mathrm{rsw}}&\Hom_{\overline F}(\fm^{r}_{\overline K}/\fm^{r+}_{\overline K},\Omega^1_F(\log)\otimes_{F}\overline F)}
\end{equation}
\end{proposition}

\section{Singular supports and characteristic cycles}

\subsection{}\label{deftrans}
Let $X$ be a smooth $k$-scheme and $C$ a closed conical subset in $\bT^*X$.   Let $f:Y\rightarrow X$ be a morphism of smooth $k$-schemes, $y$ a point of $Y$ and $\bar y\rightarrow Y$ a geometric point above $y$. We say that $f:Y\rightarrow X$ is $C$-{\it transversal at} $y$ if $\ker(df_{\bar y})\bigcap (C\times_X\bar y)\subseteq\{0\}\subseteq \bT^*_{f(\bar y)}X$, where $df_{\bar y}:\bT^*_{f(\bar y)}X\rightarrow\bT^*_{\bar y}Y$ is the canonical map. We say that $f:Y\rightarrow X$ is $C$-{\it transversal}  if it is $C$-transversal at every point of $Y$. If $f:Y\rightarrow X$ is $C$-transversal, we define $f^\circ C$ to be the scheme theoretic image of $Y\times_XC$ in $\bT^*Y$ by the canonical map $df: Y\times_X\bT^*X\rightarrow \bT^*Y$. Notice that $df:Y\times_XC\rightarrow f^\circ C$ is finite and that $f^\circ C$ is also a closed conical subset of $\bT^*Y$ (\cite[Lemma 1.2]{bei}).
Let $g:X\rightarrow Z$ be a morphism of $k$-schemes, $x$ a point of $X$ and $\bar x\rightarrow X$ a geometric point above $x$. We say that $g:X\rightarrow Z$ is $C$-transversal at $x$ if $dg_{\bar x}^{-1}(C\times_X\bar x)\subseteq\{0\}\subseteq \bT^*_{g(\bar x)}Z$, where $dg_{\bar x}:\bT^*_{g(\bar x)}Z\rightarrow \bT^*_{\bar x}X$ is the canonical map. We say that $g:X\rightarrow Z$ is $C$-{\it transversal} if it is $C$-transversal at every point of $X$. Let $(g,f):Z\leftarrow Y\rightarrow X$ be a pair of morphisms of smooth $k$-schemes. We say that $(g,f)$ is $C$-{\it transversal} if $f:Y\rightarrow Z$ is $C$-transversal and $g:Y\rightarrow Z$ is $f^\circ C$-transversal.

\subsection{}\label{microsupp}
Let $X$ be a smooth $k$-scheme and $\sK$ a bounded complex of  sheaves of $\Lambda$-modules with constructible cohomologies. We say that $\sK$ is {\it micro-supported} on a closed conical subset $C$ of $\bT^*X$ if, for any $C$-transversal pair of morphisms $(g,f):Z\leftarrow Y\rightarrow X$ of smooth $k$-schemes, the morphism $g:Y\rightarrow Z$ is locally acyclic with respect to $f^*\sK$. If there exists a smallest closed conical subset of $\bT^*X$ on which $\sK$ is micro-supported, we call it the {\it singular support} of $\sK$ and denote it by $SS(\sK)$.

\begin{theorem}[{\cite[Theorem 1.3]{bei}}]\label{beimain}
Let $X$ be a smooth $k$-scheme and $\sK$ a bounded complex of  sheaves of $\Lambda$-modules with constructible cohomologies.  The singular support $SS(\sK)$ of $\sK$ exists. If $X$ is equidimensional, each irreducible component of $SS(\sK)$ has dimension $\dim_kX$.
\end{theorem}

\subsection{}
In the following of this section, we assume that $k$ is perfect.

Let $X$ be a smooth $k$-scheme of equidimesnion $n$, $C$ a closed conical subset of $\bT^*X$ and $f:X\rightarrow \bA^1_k$ a morphism of $k$-schemes. We say that a closed point $x$ of $X$ is an {\it at most} $C$-{\it isolated characteristic point of} $f:X\rightarrow \bA^1_k$ if there exists an open neighborhood $V$ of $x\in X$ such that $f:V-\{x\}\rightarrow \bA^1_k$ is $C$-transversal. A non-zero vector of the fiber of $\bT^*\bA^1_k$ at the origin of $\bA^1_k$ induces a section $\theta_0:\bA^1_k\rightarrow \bT^*\bA^1_k$ of $\bT^*\bA^1_k$ and hence a section $\theta:X\rightarrow \bT^*\bA^1_k\times_{\bA^1_k}X$ of  $\bT^*\bA^1_k\times_{\bA^1_k}X$. The composition $df\circ\theta:X\rightarrow \bT^*X$ is a section of the bundle $\bT^*X$.

Assume that $C$ is of equidimension $n$. Let $x$ be a closed point of $X$ which is an at most $C$-isolated characteristic point of $f:X\rightarrow \bA^1_k$. For any $n$-cycle $A$ supported on $C$, there exists an open neighborhood $V$ of $x\in X$ such that the intersection $(A|_{\bT^*V},(df\circ\theta)(V))$ is a $0$-cycle supported on $\bT^*_xV$. Its degree is independent of the choice of the non-zero vector of the fiber of $\bT^*\bA^1_k$ at the origin of $\bA^1_k$ and we denote by $(A,df)_{\bT^*X,x}$ the intersection number.

\subsection{}
Let $X$ be a smooth $k$-scheme of equidimension $n$ and $\sK$ a bounded complex of sheaves of $\Lambda$-modules with constructible cohomologies. Deligne expected that there exists a unique $n$-cycle $A$ on $\bT^*X$ supported on $SS(\sK)$ satisfying the following formula of Milnor type:

For any \'etale morphism $g:W\rightarrow X$, any morphism $f:W\rightarrow \bA^1_k$, any at most $g^\circ(SS(\sK))$-isolated characteristic point $w\in W$ of $f:W\rightarrow\bA^1_k$ and any geometric point $\bar w$ above $w$, we have
\begin{equation*}
-\sum_i(-1)^i\dimtot(\rR^i\Phi_{\bar w}(g^*\sK,f))=(g^*A,df)_{\bT^*W,w}
\end{equation*}
where $\rR\Phi_{\bar w}(g^*\sK,f)$ denotes the vanishing cycle complex of $g^*\sK$ relative to $f:W\rightarrow \bA^1_k$, $\dimtot(\rR^i\Phi_{\bar w}(g^*\sK,f))$ the total dimension of the stalk $\rR^i\Phi_{\bar w}(g^*\sK,f)$ and $g^*A$ the pull-back of $A$ to $\bT^*W$. If the cycle $A$ above exists, we call it the {\it characteristic cycle} of $\sK$ and denote it by $CC(\sK)$.

\begin{theorem}[{\cite[Theorem 5.9]{cc}}]\label{ccmain}
Let $X$ be a smooth $k$-scheme of equidimension $n$ and $\sK$ a bounded complex of sheaves of $\Lambda$-modules with constructible cohomologies. The characteristic cycle $CC(\sK)$ exists.
\end{theorem}

\subsection{}
We take the notation and assumptions of Theorem \ref{ccmain}. If each cohomology of $\sK$ is locally constant, we have
\begin{equation*}
CC(\sK)=\sum_i(-1)^{i+n}(\rk_{\Lambda}\mathcal H^i(\sK))\cdot[\bT^*_XX].
\end{equation*}
If $\sK$ is a perverse sheaf on $X$, then the support of $CC(\sK)$ is $SS(\sK)$ and each coefficient of $CC(\sK)$ is a positive integer (\cite[Proposition 5.14]{cc}).
Let $Y$ be a connected and smooth $k$-curve, $E$ a Cartier divisor on $Y$, $V$ the complement of $E$ in $Y$, $h:V\rightarrow Y$ the canonical injection and $\sG$ a locally constant and constructible sheaf of $\Lambda$-modules on $V$. We have
\begin{equation*}
CC(h_!\sG)=-\rk_{\Lambda}\sG\cdot[\bT^*_YY]-\sum_{y\in E}\dimtot_y(h_!\sG)\cdot[\bT^*_yY].
\end{equation*}

\subsection{}\label{wrccss}
Let $X$ be a smooth $k$-scheme of equidimension $n$ and $D$ an integral Cartier divisor on $X$, $U$ the complement of $D$ in $X$, $j:U\rightarrow X$ the canonical injection and $\sF$ a locally constant and constructible sheaf of $\Lambda$-modules on $U$. We denote by $\xi$ be the generic point of $D$, by $\sO_K$ the henselization of $\sO_{X,\xi}$, by $F$ the residue field of $\sO_{X,\xi}$, by $K$ the fraction field of $\sO_K$, by $\eta=\Spec(K)$ the generic point of $S=\Spec(\sO_K)$, by $\overline K$ a separable closure of $K$, by $\sO_{\overline K}$ the integral closure of $\sO_K$ in $\overline K$, by $\overline F$ the residue field of $\sO_{\overline K}$, by $G_K$ the Galois group of $\overline K/K$ and by $M$ the finite generated $\Lambda$-module with continuous $G_K$-action corresponding to $\sF|_{\eta}$. After enlarging $\Lambda$, we may assume that it contains primitive $p$-th roots of unity. We have the slope decomposition and center character decompositions
\begin{equation*}
M=\bigoplus_{r\in\mathbb Q\geq 1}M^{(r)}\ \ \ \textrm{and}\ \ \ M^{(r)}=\bigoplus_{\chi\in X(r)}M^{(r)}_\chi\ \ \textrm{(for}\; r>1).
\end{equation*}
We fix a non-trivial character $\psi:\mathbb F_p\rightarrow \Lambda^{\times}$. For each $r>1$, the graded piece $\Gr^r G_K$ is abelian and killed by $p$. Each $\chi$ factors uniquely as $\Gr^r G_K\rightarrow \mathbb F_p\xrightarrow{\psi}\Lambda^{\times}$. We denote also by $\chi$ the induced character and by
\begin{equation*}
\mathrm{char}(\chi): \fm^r_{\overline K}/\fm^{r+}_{\overline K}\rightarrow \Omega^1_{\sO_{X,\xi}}\otimes_{\sO_{X,\xi}}\overline{F}
\end{equation*}
the characteristic form of $\chi$ \eqref{charform}. Let $F_{\chi}$ be a field of definition of $\mathrm{char}(\chi)$ which is a finite extension of $F$ contained in $\overline F$. The characteristic form $\mathrm{char}(\chi)$ defines a line $L_\chi$ in $\bT^*X\times_X\Spec(F_{\chi})$. Let $\overline L_{\chi}$ be the closure of the image of $L_{\chi}$ in $\bT^*X$. After removing a closed subsecheme $Z\subseteq D$ of codimension 2 in $X$, we may assume that $D_0=D-Z$ is smooth over $\Spec(k)$ and the ramification of $\sF$ along $D_0$ is non-degenerate (\cite[Definition 3.1]{wr}). Roughly speaking, the ramification of an \'etale sheaf along a smooth divisor is non-degenerate if its ramification along the divisor is controlled by its ramification at generic points of the divisor. Using Abbes and Saito's ramification theory, the characteristic cycle of $j_!\sF$ on $X_0=X-Z$ can be computed explicitly as follows (\cite[Definition 3.5]{wr} and \cite[Theorem 7.14]{cc})
\begin{equation}\label{wrcbcc}
CC(j_!\sF|_{X_0})=(-1)^n\left(\rk_{\Lambda}\sF\cdot[\bT^*_{X_0}X_0]+\dim_{\Lambda}M^{(1)}\cdot[\bT^*_{D_0}X_0]+\sum_{r>1}\sum_{\chi\in X(r)}\frac{r\cdot\dim_{\Lambda}M^{(r)}_{\chi}}{[F_{\chi}:F]}[\overline L_{\chi}]\right).
\end{equation}
The singular support $SS(j_!\sF|_{X_0})$ is the support of $CC(j_!\sF|_{X_0})$ (\cite[Proposition 3.15]{wr}). Observe that, for any point $x$ in $D_0$, the fiber $SS(j_!\sF)\times_Xx\subseteq \bT^*_xX$ is a finite union of $1$-dimensional $k(x)$-vector spaces.

\section{Total dimension divisors and Swan divisors}

\subsection{}\label{mainnotation}
In this section, let $X$ be a smooth $k$-scheme, $D$ a reduced effective Cartier divisor on $X$, $\{D_i\}_{i\in I}$ the set of irreducible components of $D$, $U$ the complement of $D$ in $X$, $j:U\rightarrow X$ the canonical injection. We assume that each $D_i$ ($i\in I$) is generically smooth over $\Spec(k)$.  Let $\mathscr F$ be a locally constant and constructible sheaf of $\Lambda$-modules on $U$.

\subsection{}
Let $\bar k$ be an algebraic closure of $k$, $X_{\bar k}=X\times_k\bar k$, $\xi_{i}$ the generic point of an irreducible component of $D_{i,\bar k}=D_i\times_{k}\bar k$, $\eta_{i}$ the generic point of the henselization $X_{\bar k,(\xi_{i})}$, $K_{i}$ the function field of  $X_{\bar k,(\xi_{i})}$ and $\overline K_i$ a separable closure of $K_i$. The restriction $\mathscr F|_{\eta_i}$ corresponds to a finitely generated $\Lambda$-module with a continuous $\mathrm{Gal}(\overline K_i/K_i)$-action.
Since the $\mathrm{Gal}(\bar k/ k)$-action on the set of irreducible components of $D_{i,\bar k}$ is transitive, the total dimension $\mathrm{dimtot}_{K_{i}}(\mathscr F|_{\eta_{i}})$ and the Swan conductor $\sw_{K_i}(\mathscr F|_{\eta_i})$ do not depend on the choice of $\bar k$ nor on the choice of the irreducible component of $D_{i,\bar k}$. We define the {\it total dimension divisor} of $j_!\mathscr F$ on $X$ and denote by $\mathrm{DT}_X(j_!\mathscr F)$ the Cartier divisor:
\begin{equation}\label{dtdiv}
\mathrm{DT}_X(j_!\mathscr F)=\sum_i \mathrm{dimtot}_{K_{i}} (\mathscr F|_{\eta_{i}})\cdot D_i.
\end{equation}
We define the {\it Swan divisor} of $j_!\mathscr F$ on $X$ and denote by $\SW_X(j_!\sF)$ the Cartier divisor:
\begin{equation}\label{swdiv}
\mathrm{SW}_X(j_!\mathscr F)=\sum_i \mathrm{sw}_{K_{i}} (\mathscr F|_{\eta_{i}})\cdot D_i.
\end{equation}
We define the {\it logarthimic total dimension divisor} of $j_!\sF$ and denote by $\LDT_X(j_!\sF)$ the divisor:
\begin{equation}\label{ldtdiv}
\LDT_X(j_!\sF)=\sum_i \left(\mathrm{sw}_{K_{i}}(\mathscr F|_{\eta_{i}})+\dim_{\Lambda}\sF\right)\cdot D_i.
\end{equation}
They have integral coefficients (Theorem \ref{xlmain}). By \eqref{swdimtotcurve}, if $X$ is $k$-curve, we have $\DT_X(j_!\sF)=\LDT_X(j_!\sF)$.

Let $k'$ be an extension of $k$ contained in $\bar k$, and $\pi:X'=X\times_kk'\rightarrow X$ the canonical projection. By definition, we have
\begin{align}
\pi^*(\mathrm{DT}_X(j_!\mathscr F))&=\mathrm{DT}_{X'}(\pi^*(j_!\mathscr F)),\label{dimtotbc}\\
\pi^*(\mathrm{SW}_X(j_!\mathscr F))&=\mathrm{SW}_{X'}(\pi^*(j_!\mathscr F)).\label{swbc}
\end{align}
In the following, we denote by $\dimtot_{D_i}(j_!\sF)$ (resp. $\sw_{D_i}(j_!\sF)$) the coefficient of $D_i$ in $\DT_X(j_!\sF)$ (resp. $\SW_X(j_!\sF)$) for simplicity (\eqref{dtdiv} and \eqref{swdiv}).

\begin{proposition}[{\cite[Proposition 3.8]{wr} and \cite[Lemma 1.22]{saito cc}}]\label{smpullback}
For any smooth morphism $f:Y\rightarrow X$, we have
\begin{align}
f^*(\DT_X(j_!\sF))&=\DT_Y(f^*j_!\sF),\label{dimtotspb}\\
f^*(\SW_X(j_!\sF))&=\SW_Y(f^*j_!\sF).\label{swspb}
\end{align}
\end{proposition}
\begin{proof}
The equality \eqref{dimtotspb} is a part of \cite[Proposition 3.8]{wr}. We need to show \eqref{swspb}. It holds for \'etale morphism by definition.
Hence, it suffices to treat the case where $Y=\bA^m_X$ for an integer $m\geq 1$. By \eqref{swbc}, we may assume that $k$ is algebraically closed. After replacing $X$ by Zariski neighborhoods of generic points of $D$, we may assume that $D$ is irreducible. We denote by $\xi$ the generic point of $D$, $K$ the fraction field of the localization of $X$ at $\xi$, by $\overline K$ a separable closure of $K$, by $G_K$ the Galois group of $\overline K$ over $K$,
by $\zeta$ the generic points of $\bA^m_D=Y\times_XD$,  by $L$ the fraction field of the localization of $Y$ at $\zeta$, by $\overline L$ a separable closure of $L$ that contains $\overline K$ and by $G_L$ the Galois group of $\overline L$ over $L$. We have a canonical group homomorphism $\phi:G_L\rightarrow G_K$. Since $f:\bA^m_X\rightarrow X$ is smooth, the ramification index of $L/K$ is $1$. Hence, $\phi$ induces a homomorphism of subgroups $\phi^r:G^r_{L,\log}\rightarrow G^r_{K,\log}$ for any $r\in \bQ_{\geq 0}$. By \cite[Lemma 1.22]{saito cc}, the induced map
\begin{equation*}
\bar \phi^r:G^r_{L,\log}/G^{r+}_{L,\log}\rightarrow G^r_{K,\log}/G^{r+}_{K,\log}
\end{equation*}
is surjective for any $r\in \bQ_{\geq 0}$. Hence, the logarithmic slope decompositions of $\sF|_{\Spec(K)}$ and $\sF|_{\Spec(L)}$ are identical. Hence, we have
\begin{equation*}
\sw_D(j_!\sF)=\sw_{\bA^m_D}(f^*(j_!\sF)).
\end{equation*}
The equality \eqref{swspb} is obtained.
\end{proof}

\subsection{}
In the following of this section, we assume that $k$ is perfect and $X$ is of equidimension $n$.

\begin{proposition}[{cf. \cite[Theorem 4.2]{HY}}]\label{hymain}
Let $Z$ be a smooth $k$-scheme and $h:Z\rightarrow X$ a $k$-morphism. We assume that $h^*D=Z\times_XD$ is an effective Cartier divisor on $Z$. Then, we have
\begin{equation}\label{DTineq}
h^*(\DT_X(j_!\sF))\geq\DT_Z(h^*(j_!\sF)).
\end{equation}
\end{proposition}
\begin{proof}
Theorem 4.2 of \cite{HY} proved \eqref{DTineq} under the condition that the base field of $X$ is algebraically closed. Here, let $\bar k$ be an algebraic closure of $k$, $X_{\bar k}=X\times_k\bar k$, $Z_{\bar k}=Z\times_k\bar k$, $\pi:X_{\bar k}\rightarrow X$ and $\pi':Z_{\bar k}\rightarrow Z$ canonical projections and $h':Z_{\bar k}\rightarrow X_{\bar k}$ the base change of $h:Z\rightarrow X$. Notice that $h\circ\pi'=\pi\circ h'$.
We have
\begin{equation}\label{bcalgclosed}
\pi^*(\DT_X(j_!\sF))=\DT_{X_{\bar k}}(\pi^*(j_!\sF))\ \ \ \textrm{and}\ \ \ \pi'^*(\DT_Z(h^*(j_!\sF)))=\DT_{Z_{\bar k}}(\pi'^*(h^*(j_!\sF))).
\end{equation}
By \cite[Theorem 4.2]{HY}, we have
\begin{equation}\label{hytheorem}
h'^*(\DT_{X_{\bar k}}(\pi^*(j_!\sF)))\geq\DT_{Z_{\bar k}}(\pi'^*(h^*(j_!\sF))).
\end{equation}
By \eqref{bcalgclosed} and \eqref{hytheorem}, we get
\begin{equation*}
\pi'^*(h^*(\DT_X(j_!\sF)))\geq \pi'^*(\DT_Z(h^*(j_!\sF))).
\end{equation*}
It implies that \eqref{DTineq} holds since a Cartier divisor $E$ on $Z$ is effective if and only if $E_{\bar k}=E\times_k\bar k$ is effective on $Z_{\bar k}$.
\end{proof}

This inequality generalizes a similar result due to Saito which requires more geometric and ramification conditions (\cite[Corollary 3.9.1]{wr}). Earlier than Saito's result, Matsuda obtained it for the case where $Z$ is a smooth $k$-curve and $\sF$ has rank $1$ (\cite{mat}, cf. \cite[Proposition 2.7]{ks}). The following proposition implies that the singular support gives the criterion for
the equality of both sides of the inequality \eqref{DTineq}.

\begin{proposition}[{\cite[Corollary 3.9.2]{wr}}]\label{saitomain}
Let $C$ be a smooth $k$-curve and $g:C\rightarrow X$ a qusai-finite morphism. We assume that $g^*D=C\times_XD$ is an effective Cartier divisor on $C$, that $D$ is smooth and the ramification of $\sF$ along $D$ is non-degenerate at each point of $g(C)\bigcap D$, and that $g:C\rightarrow X$ is $SS(j_!\sF)$-transversal. Then, we have
\begin{equation*}
g^*(\DT_X(j_!\sF))=\DT_C(g^*(j_!\sF)).
\end{equation*}
\end{proposition}

\section{Bounding the pull-back of the Swan divisor}

\subsection{}
In this section, we assume that $k$ is perfect. Let $X$ be a connected and smooth $k$-scheme, $D$ an effective Cartier divisor on $X$ with simple normal crossing, $\{D_i\}_{i\in I}$ the set of irreducible components of $D$, $U$ the complement of $D$ in $X$, $j:U\rightarrow X$ the canonical injection and $\mathscr F$ a locally constant and constructible sheaf of $\Lambda$-modules on $U$.

\begin{proposition}\label{propimm}
Let $C$ be a smooth $k$-curve and $g:C\rightarrow X$ a quasi-finite morphism. We assume that $g^*D=C\times_XD$ is a Cartier divisor on $C$. Then, we have
\begin{equation}\label{immpullback}
g^*(\SW_X(j_!\sF))\geq\SW_C(g^*(j_!\sF)).
\end{equation}
\end{proposition}
\begin{proof}
The quasi-finite morphism $g:C\rightarrow X$ is a composition of the graph $\Gamma_g:C\rightarrow C\times_kX$ and the second projection $\mathrm{pr}_2:C\times_kX\rightarrow X$. Then, by Proposition \ref{smpullback}, we have
\begin{equation*}
\SW_{C\times_kX}(\mathrm{pr}_2^*(j_!\sF))=\mathrm{pr}_2^*(\SW_X(j_!\sF)).
\end{equation*}
Hence, it suffices to treat the case where $g:C\rightarrow X$ is a closed immersion. By a similar d\'evissage as in the proof of Proposition \ref{hymain} for Swan divisors, we may assume that $k$ is algebraically closed. This is a local problem for the Zariski topology of $C$ and $X$. We may assume that $X$ is affine and that $C$ and $D$ intersect at a single point $x\in X$.

Let $D_1, D_2,\cdots, D_m$ be irreducible components of $D$ that contain $x$,  $f_i$ $(i=1,2,\cdots, m)$ an element of $\Gamma(X,\sO_X)$ that defines $D_i$, and $t\in \Gamma(C,\sO_C)$ a local coodinate at $x$. For each $i=1,2,\cdots, m$, we denote by $\alpha_i$ the multiplicity of $g^*D_i$ at $x$. After replacing $X$ by a Zariski neighborhood of $x$, we may normalize each $f_i$ by $g^*f_i=t^{\alpha_i}\in \sO_{C,x}$  $(i=1,2,\cdots, m)$. Let $\cB$ be the set of $m$-ples of positive integers $\underline\beta=\{\beta_i\}_{i=1,2,\cdots m}$ such that each $\beta_i$ is co-prime to $p\cdot\prod_{i=1}^m\alpha_i$. Let
\begin{equation*}
X_{\underline \beta}=\Spec(\sO_X[T_1,\cdots, T_m]/(T_1^{\beta_1}-f_1,\cdots, T_m^{\beta_m}-f_m))
\end{equation*}
be a cover of $X$ tamely ramified along each $D_i$
and we denote by $h_{\underline\beta}:X_{\underline \beta}\rightarrow X$ the canonical projection. Notice that $X_{\ub}$ is a smooth $k$-scheme. We put $N_{\beta}=\prod_{i=1}^m\beta_i$. Let
\begin{align*}
C_{\ub}&=\Spec(\sO_C[T]/(T^{N_{\ub}}-t)).
\end{align*}
We have a commutative diagram
\begin{equation*}
\xymatrix{\relax C_{\underline\beta}\ar[d]_{g_{\ub}}\ar[r]^-(0.5){[N_{\ub}]}&C\ar[d]^g\\
X_{\underline\beta}\ar[r]_-(0.5){h_{\underline\beta}}&X}
\end{equation*}
where $[N_{\ub}]:C_{\ub}\rightarrow C$ is the canonical cyclic cover of degree $N_{\ub}$ tamely ramified at $x$ and $g_{\ub}:C_{\ub}\rightarrow X_{\ub}$ is given by
\begin{align*}
g_{\ub}^*:\sO_X[T_1,\cdots, T_m]/(T_1^{\beta_1}-f_1,\cdots, T_m^{\beta_m}-f_m)\rightarrow \sO_C[T]/(T^{N_{\ub}}-t),
\ \ \ T_i\mapsto (T^{N_{\ub}})^{\frac{\alpha_i }{\beta_i}}.
\end{align*}
Notice that the pre-image of $x\in X$ in $C_{\ub}$ is a single point. We denote it by $x'$. Let $D_{i,\ub}$ be the smooth divisor $(T_i)=(D_i\times_XX_{\ub})_{\mathrm{red}}$ of $X_{\ub}$ and $D_{\ub}$ the simple normal crossing divisor $(T_1\cdots T_m)=(D\times_XX_{\ub})_{\mathrm{red}}$ of $X_{\ub}$. We have
\begin{align*}
m_{x'}((h_{\ub}\circ g_{\ub})^*D_i)=N_{\ub}\cdot\alpha_i\ \ \ &\textrm{and}\ \ \
m_{x'}( (h_{\ub}\circ g_{\ub})^*D)=N_{\ub}\cdot\sum_{i=1}^m\alpha_i,\\ m_{x'}( g_{\ub}^*D_{i,\ub})=N_{\ub}\cdot\frac{\alpha_i}{\beta_i}\ \ \ &\textrm{and}\ \ \  m_{x'}( g_{\ub}^*D_{\ub})=N_{\ub}\cdot\sum^m_{i=1}\frac{\alpha_i}{\beta_i}.
\end{align*}

Applying \cite[Theorem 4.3]{HY} (cf. Proposition \ref{hymain}) to the sheaf $h_{\ub}^*(j_!\sF)$ on $X_{\ub}$ and the morphism $g_{\ub}:C_{\ub}\rightarrow X_{\ub}$, we have
\begin{equation}\label{hyapply}
\sum^m_{i=1}N_{\ub}\cdot\frac{\alpha_i}{\beta_i}\cdot\dimtot_{D_{i,\ub}}(h_{\ub}^*(j_!\sF)))\geq \dimtot_{x'}((h_{\ub}\circ g_{\ub})^*(j_!\sF))
\end{equation}
By \eqref{swdimtotcomp} and \eqref{swdimtotcurve}, we have
\begin{align}
\sw_{D_{i,\ub}}(h_{\ub}^*(j_!\sF)))+\rk_{\Lambda}\sF &\geq \dimtot_{D_{i,\ub}}(h_{\ub}^*(j_!\sF))),\label{compareX}\\
\sw_{x'}((h_{\ub}\circ  g_{\ub})^*(j_!\sF))+\rk_{\Lambda}\sF&=\dimtot_{ x'}((h_{\ub}\circ  g_{\ub})^*(j_!\sF)).\label{compareC}
\end{align}
Since each $\beta_i$ $(i=1,\cdots, m)$ is co-prime to $p$, the covering $h_{\ub}:X_{\ub}\rightarrow X$ is tamely ramified along the divisor $D$. By Proposition \ref{unramtame}, we have
\begin{align}
\sw_{D_{i,\ub}}(h_{\ub}^*(j_!\sF)))&=\beta_i\cdot\sw_{D_i}(j_!\sF),\label{tameswX}\\
\sw_{x'}((h_{\ub}\circ  g_{\ub})^*(j_!\sF))&=N_{\ub}\cdot \sw_x(g^*(j_!\sF)).\label{tameswC}
\end{align}
By \eqref{hyapply}, \eqref{compareX}, \eqref{compareC}, \eqref{tameswX} and \eqref{tameswC}, we have
\begin{equation}
\sum^m_{i=1}N_{\ub}\cdot\frac{\alpha_i}{\beta_i}\cdot(\beta_i\cdot\sw_{D_i}(j_!\sF)+ \rk_{\Lambda}\sF)\geq N_{\ub}\cdot \sw_x(g^*(j_!\sF))+\rk_{\Lambda}\sF.
\end{equation}
Divide both sides by $N_{\ub}$, we obtain
\begin{equation*}
\sum^m_{i=1}\alpha_i\cdot\sw_{D_i}(j_!\sF)+\left(\sum^m_{i=1}\frac{\alpha_i}{\beta_i}-\frac{1}{N_{\ub}}\right)\cdot\rk_{\Lambda}\sF\geq \sw_x(g^*(j_!\sF)).
\end{equation*}
This inequality holds for any element $\ub$ of $\cB$. Passing $\beta_i\rightarrow +\infty$ $(i=1,\cdots, m)$, we get
\begin{equation*}
\sum^m_{i=1}\alpha_i\cdot\sw_{D_i}(j_!\sF)\geq \sw_x(g^*(j_!\sF)).
\end{equation*}
It gives the inequality \eqref{immpullback}.
\end{proof}

\begin{proposition}\label{findcurve}
Assume that $X$ is affine and of dimension $n\geq 2$ and $D$ is irreducible and smooth over $\Spec(k)$. Let $\beta$ be a positive integer co-prime to $p$, $f$ an element of $\Gamma(X,\sO_X)$ that defines $D$,
\begin{equation*}
X_{\beta}=\Spec(\sO_X[T]/(T^\beta-f))
\end{equation*}
a cyclic cover of $X$ of degree $\beta$ tamely ramified along $D$, $h_{\beta}:X_{\beta}\rightarrow X$ the canonical projection, $D_{\beta}$ the smooth divisor $(T)=(D\times_XX_{\beta})_{\mathrm{red}}$ of $X_{\beta}$. Let $B$ a closed conical subset of  $\bT^*X_{\beta}\times_{X_{\beta}}D_{\beta}$ of equidimension $n$ such that, for any $z\in D_{\beta}$, the fiber $B_z\subseteq \bT^*_zX_{\beta}$ is non-empty. Then, we can find a Zariski open dense subset $W_{\beta}$ of $D_{\beta}$ such that, for any closed point $x'$ of $W_{\beta}$, we can find a smooth $k$-curve $C$ and an immersion $g_{\beta}:C\rightarrow X_{\beta}$ such that
\begin{itemize}
\item[(i)] $C\bigcap D_{\beta}=x'$ and $m_{x'}(g^*_{\beta}D_{\beta})=1$;
\item[(ii)] $g_{\beta}:C\rightarrow X_{\beta}$ is $B$-transversal at $x'$;
\item[(iii)] the composition $h_{\beta}\circ g_{\beta}:C\rightarrow X$ is also an immersion.
\end{itemize}
\end{proposition}

\begin{proof}
After replacing $X$ by an open dense subset of the generic point of $D$, we may assume that, for any point $z$ of $D_{\beta}$, the fiber $B_z\subseteq \bT^*_zX_{\beta}$ is a union of $1$-dimensional $k(z)$-vector spaces.  Let $x'$ be a closed point of $D_{\beta}$ and $x=h_{\beta}(x')$. Since $h_{\beta}|_{D_{\beta}}:D_{\beta}\rightarrow D$ is isomorphic, we have $k(x)=k(x')$. We denote by $\overline T$ the image of $T\in\fm_{X_{\beta},x'}$ in $\bT^*_{x'}X_{\beta}$ and by $\bar f$ the image of $f\in\fm_{X,x}$ in $\bT^*_xX$. The canonical map $dh_{\beta,x'}:\bT^*_xX\rightarrow \bT^*_{x'}X_{\beta}$ induces an exact sequence of $k(x)$-vector spaces
\begin{equation*}
0\rightarrow\mathrm{span}\{\bar f\}\rightarrow\bT^*_xX\xrightarrow{{dh_{\beta,x'}}}\bT^*_{x'}X_{\beta}\rightarrow M\rightarrow 0,
\end{equation*}
where $M$ is of rank $1$ such that the image of $\overline T$ in $M$ is non-zero. We denote by $H_X$ the image of $dh_{\beta,x'}$ in $T^*_{x'}X_{\beta}$ and we have $\mathrm{span}\{\overline T\}+H_X=T^*_{x'}X_{\beta}$.
We denote by $H^{\vee}_X\subseteq \bT_{x'}X_{\beta}$ the dual line of $H_X$. In fact, $H^{\vee}_X=\ker(\bT h_{\beta,x'})$, where $\bT h_{\beta,x'}:\bT_{x'} X_{\beta}\rightarrow \bT_x X$ is the dual of $dh_{\beta,x'}:\bT^*_xX\rightarrow \bT^*_{x'}X_{\beta}$.

By \cite[III, 9.7.8]{EGA4}, we can find an open dense subset $V_{\beta}$ of $D_{\beta}$ and a positive integer $N$ such that, for each closed point $x'$ of $V_{\beta}$, the fiber $B_{x'}$ is a finite union of $1$-dimensional vector spaces $\{L_{x',(i)}\}_{i\in I}$, where the cardinality of $I$ is less than $N$. There exists an integer $N'$ such that, for any field extension $k'/k$ with $[k':k]\geq N'$, any union of $(N+2)$'s hyperplanes of $k'^n$ will not cover the whole vector space (If $k$ has infinitely many elements, $k'=k$ satisfies the condition). Let $W_{\beta}$ be a Zariski open dense subset of $V_{\beta}$ such that, for any closed point $x'$ of $W_{\beta}$, we have $[k(x'):k]\geq N'$. Let $L^{\vee}_{x',i}\subseteq \bT_{x'}X_{\beta}$ be the dual hyperplane of $L_{x',i}$ and $\bT_{x'}D_{\beta}$ the stalk of the tangent bundle of $\bT D_{\beta}$ at $x'$.
For any closed point $x'$ of $W_{\beta}$, the complement $\bT_{x'} X_{\beta}-((\bigcup_{i\in I}L_{x',i}^{\vee})\bigcup \bT_{x'}D_{\beta}\bigcup H^{\vee}_X)$ is non-empty. Hence,
we can find a non-zero vector $\lambda\in \bT_{x'} X_{\beta}-((\bigcup_{i\in I}L_{x',i}^{\vee})\bigcup \bT_{x'}D_{\beta}\bigcup H^{\vee}_X)$. We denote by $H_\lambda\subseteq \bT^*_{x'}X_{\beta}$ the dual hyperplane of $\lambda\in\bT_{x'}X_{\beta}$. Since $\lambda\notin \bT_{x'}D_{\beta}$, we have $\mathrm{span}\{\overline T\}+H_{\lambda}=\bT^*_{x'}X_{\beta}$. Since $\lambda\not\in H^{\vee}_X$, we have $H_X\neq H_\lambda$. Let $\{\bar t_2,\cdots,\bar t_{n-1}\}$ be a base of the rank $n-2$ vector space $H_\lambda\bigcap H_X$ and $\bar t_{1}$ a non-zero vector of $H_X$ such that $\{\bar t_1,\cdots,\bar t_{n-1}\}$ is base of $H_X$ and that $\{\bar t_1-\overline T,\bar t_2\cdots,\bar t_{n-1}\}$ is a base of $H_\lambda$.  Let $t_1,\cdots, t_{n-1}$ be liftings of $\bar t_1,\cdots,\bar t_{n-1}$ in the maximal ideal $\fm_{X,x}$ of $\sO_{X,x}$. Observe that $f, t_1,\cdots, t_{n-1}$ is a regular system of parameters of $\sO_{X,x}$ and $T, t_1-T,t_2,\cdots, t_{n-1}$ is a regular system of parameters of $\sO_{X_{\beta},x'}$.

In a Zariski neighborhood of $x'\in X_{\beta}$, we define a $k$-curve $C$ by the ideal
\begin{equation*}
(t_1-T,t_2,\cdots, t_{n-1})\subseteq\sO_{X_{\beta},x'}=\sO_{X,x}[T]/(T^\beta-f).
\end{equation*}
 It is smooth at $x'$ and $\bT_{x'}C$ is spanned by $\lambda\in\bT_{x'}X_{\beta}$. We denote by $g:C\rightarrow X_{\beta}$ the canonical injection. Since $\bT_{x'}C\not\subseteq \bT_{x'}D_{\beta}$, we have
\begin{equation*}
m_{x'}(g^*_{\beta} D_{\beta})=1\ \ \ \textrm{and}\ \ \ m_{x'}((h_{\beta}\circ g_{\beta})^*D)=\beta.
\end{equation*}
 Since $\bT_{x'}C\not\subseteq\bigcup_{i\in I}L_{x',i}^{\vee}$, the map $g_{\beta}:C\rightarrow X_{\beta}$ is $B$-transversal at $x'$.
Since
\begin{equation*}
\sO_{C,x'}=\sO_{X,x}[T]/(T^\beta-f, t_1-T,t_2,\cdots,t_{n-1})\xrightarrow{\sim}\sO_{X,x}/(t_1^\beta-f,t_2,\cdots, t_{n-1}),
\end{equation*}
in a Zariski neighborhood of $x\in X$, the composition $h_{\beta}\circ g_{\beta}:C\rightarrow X$ is a closed immersion.
Hence, the Zariski open dense subset $W_{\beta}$ of $D_{\beta}$ and the curve $C$ satisfy our conditions.
\end{proof}

\begin{proposition}\label{conj2strong}
Assume that $X$ is of dimension $n\geq 2$ and that $D$ is irreducible and smooth over $\Spec(k)$. Let $\beta$ be a positive integer co-prime to $p$, $f$ an element of $\Gamma(X,\sO_X)$ that defines $D$,
\begin{equation*}
X_{\beta}=\Spec(\sO_X[T]/(T^\beta-f))
\end{equation*}
a cyclic cover of $X$ of degree $\beta$ tamely ramified along $D$, $h_{\beta}:X_{\beta}\rightarrow X$ the canonical projection, $D_{\beta}$ the smooth divisor $(T)=(D\times_XX_{\beta})_{\mathrm{red}}$ of $X_{\beta}$, $U_{\beta}$ the complement of $D_{\beta}$ in $X_{\beta}$. We can find a Zariski open dense subset $W_{\beta}$ of $D_{\beta}$ such that the ramification of $\sF|_{U_{\beta}}$ along $W_{\beta}$ is non-degenerate (\ref{wrccss}) and that, for any closed point $x'$ of $W_{\beta}$, there exists an immersion $g_{\beta}:C\rightarrow X_{\beta}$ satisfying (Proposition \ref{findcurve})
\begin{itemize}
\item[(i)] $C\bigcap D_{\beta}=x'\in D_{\beta,0}$ and $m_{x'}(g^*_{\beta}D_{\beta})=1$;
\item[(ii)] $g_{\beta}:C\rightarrow X_{\beta}$ is $SS(h_{\beta}^*(j_!\sF))$-transversal at $x'$;
\item[(iii)] the composition $h_{\beta}\circ g_{\beta}:C\rightarrow X$ is also an immersion.
\end{itemize}
Let $x'$ be a closed point of  $W_{\beta}$ and $g_{\beta}:C\rightarrow X_{\beta}$ an immersion satisfying (i)--(iii). Then, we have
\begin{equation}\label{limitsw}
\sw_D(j_!\sF)\geq \frac{1}{\beta}\cdot\sw_{x'}((h_{\beta}\circ g_{\beta})^*(j_!\sF))\geq \sw_D(j_!\sF)-\frac{1}{\beta}\cdot \rk_\Lambda\sF.
\end{equation}
\end{proposition}

\begin{proof}
Applying \cite[Corollary 3.9.2]{wr} (cf. Proposition \ref{saitomain}) to the sheaf $h_{\beta}^*(j_!\sF)$ on $X_{\beta}$ and the morphism $g_{\beta}:C\rightarrow X_{\beta}$, we have
\begin{equation}\label{applysaitomain}
\dimtot_{D_{\beta}}(h_{\beta}^*(j_!\sF))=\dimtot_{x'}((h_{\beta}\circ g_{\beta})^*(j_!\sF)).
\end{equation}
By \eqref{swdimtotcomp} and \eqref{swdimtotcurve}, we have
\begin{align}
\dimtot_{D_{\beta}}(h_{\beta}^*(j_!\sF))&\geq \sw_{D_{\beta}}(h_{\beta}^*(j_!\sF)),\label{dimtotgeqsw}\\
\dimtot_{x'}((h_{\beta}\circ g_{\beta})^*(j_!\sF))&=\sw_{x'}((h_{\beta}\circ g_{\beta})^*(j_!\sF))+\rk_{\Lambda}\sF. \label{dimtoteqswrk}
\end{align}
Since $\beta$ is co-prime to $p$, we have (Proposition \ref{unramtame})
\begin{equation}\label{swcomptame}
\sw_{D_{\beta}}(h_{\beta}^*(j_!\sF))=\beta\cdot \sw_{D}(j_!\sF).
\end{equation}
By \eqref{applysaitomain}, \eqref{dimtotgeqsw}, \eqref{dimtoteqswrk} and \eqref{swcomptame}, we have
\begin{equation}\label{swrkgeqbsw}
\sw_{x'}((h_{\beta}\circ g_{\beta})^*(j_!\sF))+\rk_{\Lambda}\sF\geq \beta\cdot \sw_{D}(j_!\sF).
\end{equation}
Applying Proposition \ref{propimm} to the sheaf $j_!\sF$ on $X$ and the injection $h_{\beta}\circ g_{\beta}:C\rightarrow X$, we have
\begin{equation}\label{bswgeqsw}
\beta\cdot \sw_{D}(j_!\sF)\geq\sw_{x'}((h_{\beta}\circ g_{\beta})^*(j_!\sF)).
\end{equation}
Divide both sides of \eqref{swrkgeqbsw} and \eqref{bswgeqsw} by $\beta$, we get
\begin{equation*}
\sw_D(j_!\sF)\geq \frac{1}{\beta}\cdot\sw_{x'}((h_{\beta}\circ g_{\beta})^*(j_!\sF))\geq \sw_D(j_!\sF)-\frac{1}{\beta}\cdot\rk_\Lambda\sF.
\end{equation*}
\end{proof}

\begin{theorem}[{cf. Conjecture \ref{ekconjb}}]\label{thmconjb}
Assume that the divisor $D$ is irreducible and smooth over $\Spec(k)$.
Let $\mathcal{PC}_0(X)$ be the set of pairs $(g: C\rightarrow X,\, x\in C)$ where $C$ is a smooth $k$-curve, $g:C\rightarrow X$ is an immersion such that $g^*D=C\times_XD$ is a Cartier divisor on $C$, and $x$ is a closed point of  $C\times_XD$. Then, we have
\begin{equation}\label{thmconjbform}
\sup_{\mathcal{PC}_0(X)}\frac{\sw_x(g^*(j_!\sF))}{m_x(g^*D)}=\sw_D(j_!\sF).
\end{equation}
\end{theorem}

\begin{proof}
When $X$ is a $k$-curve, the theorem is trivial. When $\dim_k X\geq 2$, Proposition \ref{findcurve} and Proposition \ref{conj2strong} implies that, for any positive integer $\beta$ co-prime to $p$, we can find a smooth $k$-curve $C$ and an immersion $g:C\rightarrow X$ such that
$x=C\bigcap D$ is a point, that $m_x(g^*D)=\beta$ and that
\begin{equation*}
\sw_D(j_!\sF)\geq \frac{1}{\beta}\cdot\sw_{x}(g^*(j_!\sF))\geq \sw_D(j_!\sF)-\frac{1}{\beta}\cdot\rk_\Lambda\sF.
\end{equation*}
Taking sufficient large $\beta$, we obtain \eqref{thmconjbform}.

In fact, $\mathcal{PC}_0(X)$ can be considered as a subset of $\mathcal{PC}(X)$ defined in Conjecture \ref{ekconjb}. Hence, Conjecture \ref{ekconjb} is deduced by this theorem under the condition that $X$ is smooth and $D$ is irreducible and smooth.
\end{proof}

\begin{theorem}\label{proofconj1}
Let $Z$ be a smooth $k$-scheme and $h:Z\rightarrow X$ a morphism such that $Z\times_XD$ is a Cartier divisor on $Z$. Then, we have
\begin{equation}\label{formulaconj1}
h^*(\SW_X(j_!\sF))\geq\SW_Z(h^*(j_!\sF)).
\end{equation}
\end{theorem}
\begin{proof}
The case where $\dim_kZ=1$ is proved by Proposition \ref{propimm}. We only consider the case where $\dim_k Z\geq 2$.
The morphism $f:Z\rightarrow X$ is a composition of the graph morphism $\Gamma_f:Z\rightarrow Z\times_kX$ and the second projection $\mathrm{pr}_2:Z\times_kX\rightarrow X$. By Proposition \ref{smpullback}, we have
\begin{equation*}
\SW_{Z\times_kX}(\mathrm{pr}_2^*(j_!\sF))= \mathrm{pr}_2^*(\SW_{X}(j_!\sF)).
\end{equation*}
Hence, it suffices to treat the case where $h:Z\rightarrow X$ is a closed immersion.
This is a local problem for the Zariski topology of $Z$, we may further assume that $E=(Z\times_XD)_{\mathrm{red}}$ is irreducible and smooth over $\Spec(k)$. Let $\beta$ be a positive integer co-prime to $p$. Applying Proposition \ref{conj2strong} to the sheaf $h^*(j_!\sF)$ on $X$, we can find a smooth $k$-curve $C$ and an immersion $g:C\rightarrow Z$ such that $z=C\bigcap E$ is a closed point of $C$, that $m_z(g^*E)=\beta$ and that
\begin{equation}\label{appconj2strong}
\frac{1}{\beta}\cdot\sw_{z}((h\circ g)^*(j_!\sF))\geq \sw_E(h^*(j_!\sF))- \frac{1}{\beta}\cdot\rk_\Lambda\sF.
\end{equation}
Applying Proposition \ref{propimm} to the sheaf $j_!\sF$ on $X$ and the morphism $h\circ g:C\rightarrow X$, we have
\begin{equation}\label{appconj1curve}
m_z(g^*(h^*(\SW_X(j_!\sF))))\geq \sw_z((h\circ g)^*(j_!\sF)).
\end{equation}
By \eqref{appconj2strong} and \eqref{appconj1curve}, we have
\begin{equation*}
\frac{1}{\beta}\cdot m_z(g^*(h^*(\SW_X(j_!\sF))))\geq \sw_E(h^*(j_!\sF))- \frac{1}{\beta}\cdot\rk_\Lambda\sF,
\end{equation*}
i.e., we have
\begin{equation*}
h^*(\SW_X(j_!\sF))\geq \SW_Z(h^*(j_!\sF))-\frac{1}{\beta}\cdot\rk_\Lambda\sF\cdot E.
\end{equation*}
Passing $\beta\rightarrow +\infty$, we obtain \eqref{formulaconj1}.
\end{proof}

\begin{theorem}[{cf. Conjecture \ref{ekconja}}]\label{thmconja}
Let $R$ be a Cartier divisor on $X$ supported on $D$. Then, the following there conditions are equivalent:
\begin{itemize}
\item[(1)] The ramification of $\sF$ along $D$ is bounded by $R$ in the sense of Deligne (\ref{ramdeligne});
\item[(2)]
We have
\begin{equation*}
R\geq \SW_X(j_!\sF);
\end{equation*}
\item[(3)]
For any morphism $h:Z\rightarrow X$ of smooth $k$-schemes such that $Z\times_XD$ is a Cartier divisor on $Z$, we have
\begin{equation*}
h^*R\geq \SW_{Z}(h^*(j_!\sF)).
\end{equation*}
\end{itemize}
\end{theorem}

\begin{proof}
Proposition \ref{propimm} implies that the ramification of $\sF$ along $D$ is bounded by the Swan divisor $\SW_X(j_!\sF)$ in the sense of Deligne, i.e., $(2)\Rightarrow  (1)$. Proposition \ref{conj2strong} and Theorem \ref{thmconjb} implies that $\SW_X(j_!\sF)$ is the sharp bound of the ramification of $\sF$ along $D$ in the sense of Deligne, i.e., $(1)\Rightarrow (2)$.
The equivalence of (2) and (3) is a deduction of Theorem \ref{proofconj1}.

This theorem implies Conjecture \ref{ekconja} under the condition that $X$ is smooth and $D$ is a divisor with simple normal crossing.
\end{proof}

\section{Criterion for the equality of the Swan conductor and the total dimension}

\subsection{}
In this section, we assume that $k$ is perfect. Let $X$ be a smooth $k$-scheme, $D$ an irreducible Cartier divisor on $X$, $\xi$ the generic point of $D$, $K$ the function field of the henselization $X_{(\xi)}$, $\overline K$ an separable closure of $K$, $\eta=\Spec(K)$ the generic point of  $X_{(\xi)}$, $G_K$ the Galois group of $\overline K/K$, $P_K$ the wild inertia subgroup of $G_K$, $U$ the complement of $D$ in $X$, $j:U\rightarrow X$ the canonical injection and $\sF$ a locally constant and constructible sheaf of $\Lambda$-modules on $U$.

\begin{proposition}\label{sw=dimtot}
The following three conditions are equivalent:
\begin{itemize}
\item[(i)]
The slope decomposition and logarithmic slope decomposition of $\sF|_{\eta}$ associated to the continuous $G_K$-action are the same;
\item[(ii)]
We have
\begin{equation}\label{sweqdimtotbyss}
\sw_D(j_!\sF)=\dimtot_D(j_!\sF);
\end{equation}
\item[(iii)]
After replacing $X$ by an open neighborhood of $\xi$ such that $D$ is smooth over $\Spec(k)$, the conormal bundle $\bT^*_DX$ is not contained in the singular support $SS(j_!\sF)$.
\end{itemize}
\end{proposition}
\begin{proof}
The equivalence of (i) and (ii) is proved in \ref{swdimtot}.

(iii) $\Rightarrow $ (ii). By \eqref{dimtotbc} and \eqref{swbc}, we may assume that $k$ is algebraically closed. This is a local problem for the Zariski topology of $X$. We may replace $X$ by an affine open neighborhood of $\xi$ such that $D$ is smooth over $\Spec(k)$. Let $f$ be an element of $\Gamma(X,\sO_X)$ that defines $D$,  $\beta\geq 2$ a positive integer co-prime to $p$,
\begin{equation*}
X_{\beta}=\Spec(\sO_X[T]/(T^\beta-f))
\end{equation*}
 a cyclic cover of $X$ of degree $\beta$ tamely ramified along $D$,  $h_\beta:X_{\beta}\rightarrow X$ the canonical projection, $U_{\beta}=h_{\beta}^{-1}(U)$ and $D_\beta$ the smooth divisor $(T)=(D\times_XX_\beta)_{\mathrm{red}}$. Notice that $h_{\beta}|_{D_{\beta}}:D_{\beta}\rightarrow D$ is an isomorphism. After shrinking  $X$ again, we may assume that the ramification of $\sF$ (resp. $\sF|_{U_{\beta}}$) along $D$ (resp. $D_{\beta}$) is non-degenerate (\ref{wrccss}) and that, for any point $x'$ of $D_{\beta}$ with the image $x$ in $D$, the fiber of the conormal bundle $\bT^*_DX\times_Dx$ is not contained in $SS(j_!\sF)\times_Xx$. The later implies that $h_{\beta}:X_{\beta}\rightarrow X$ is $SS(j_!\sF)$-transversal. Notice that $SS(h^*_{\beta}(j_!\sF))\subseteq h^\circ_{\beta}(SS(j_!\sF))$ (\cite[Lemma 2.2]{bei}) and, for any point $x'\in D_{\beta}$, the fiber $h_{\beta}^\circ(SS(j_!\sF))\times_{X_{\beta}}x'$ is a union of $1$-dimensional $k(x')$-vector spaces in $\bT^*_{x'}X_{\beta}$ (\ref{deftrans}).

By Proposition \ref{findcurve}, we can find a closed point $x'$ of $D_{\beta}$ with image $x$ in $D$, a smooth $k$-curve $C$ and an immersion $g_{\beta}:C\rightarrow X_{\beta}$ such that
\begin{itemize}
\item[(i)] $C\bigcap D_{\beta}=x'$ and $m_{x'}(g^*_{\beta}D_{\beta})=1$ (i.e., $m_{x'}((h_{\beta}\circ g_{\beta})^*D)=\beta$);
\item[(ii)] $g_{\beta}:C\rightarrow X_{\beta}$ is $h_{\beta}^\circ(SS(j_!\sF))$-transversal at $x'$;
\item[(iii)] the composition $h_{\beta}\circ g_{\beta}:C\rightarrow X$ is also an immersion.
\end{itemize}

By Proposition \ref{conj2strong}, we have
\begin{equation}\label{usestrong}
\sw_D(j_!\sF)\geq \frac{1}{\beta}\cdot\sw_{x'}((h_{\beta}\circ g_{\beta})^*(j_!\sF))\geq \sw_D(j_!\sF)-\frac{1}{\beta}\cdot\rk_\Lambda\sF.
\end{equation}
Since $h_{\beta}:X_{\beta}\rightarrow X$ is $SS(j_!\sF)$-transversal and $g_{\beta}:C\rightarrow X_{\beta}$ is $h_{\beta}^\circ(SS(j_!\sF))$-transversal at $x'$, we have $h_{\beta}\circ g_{\beta}:C\rightarrow X$ is $SS(j_!\sF)$-transversal at $x=g_{\beta}(x')$. Applying \cite[Corollary 3.9.2]{wr} (cf. Theorem \ref{saitomain}) to the sheaf $j_!\sF$ on $X$ and the closed immersion $h_{\beta}\circ g_{\beta}:C\rightarrow X$, we obtain
\begin{equation}\label{usesaito2}
\dimtot_{x'}((h_{\beta}\circ g_{\beta})^*(j_!\sF))=\beta\cdot\dimtot_D(j_!\sF).
\end{equation}
Since $C_{\beta}$ is a smooth $k$-curve, we have  \eqref{swdimtotcurve}
\begin{equation}\label{applycurveswdt}
\dimtot_{x'}((h_{\beta}\circ g_{\beta})^*(j_!\sF))=\sw_{x'}((h_{\beta}\circ g_{\beta})^*(j_!\sF))+\rk_\Lambda\sF.
\end{equation} By \eqref{usestrong}, \eqref{usesaito2} and \eqref{applycurveswdt}, we get
\begin{equation*}
\sw_D(j_!\sF)\geq \dimtot_D(j_!\sF)-\frac{1}{\beta}\cdot\rk_\Lambda\sF\geq \sw_D(j_!\sF)-\frac{1}{\beta}\cdot\rk_\Lambda\sF.
\end{equation*}
Passing $\beta\rightarrow\infty$, we obtain \eqref{sweqdimtotbyss}.

(i) $\Rightarrow$ (iii). We denote by $M$ the finitely generated $\Lambda$-module with a continuous $G_K$-action corresponding to $\sF|_{\eta}$. Since $M^{(r)}=0$ for $r<1$, we have $M^{P_K}=M^{(0)}_{\log}=M^{(0)}=0$. Hence, $M$ is purely wildly ramified. Fix an rational number $r>1$ such that $M^{(r)}\neq 0$. By enlarging $\Lambda$, we may assume that it contains $p$-th roots of unity. Let $M^{(r)}=\bigoplus_{\chi\in X(r)}M^{(r)}_{\chi}$ be the center character decomposition \eqref{ccdecomp}. Since the two filtrations are the same, the action of $\Gr^r_{\log}G_K$ on some $M^{(r)}_{\chi}\neq 0$ is faithful. Hence, for every non-trivial character $\chi:\Gr^r G_K\rightarrow \Lambda^{\times}$ appears in $M^{(r)}$, the composition $\chi\circ\theta_r:\Gr^r_{\log}G_K\rightarrow \Lambda^{\times}$ is non-trivial, where $\theta_r:\Gr^r_{\log}G_K\rightarrow \Gr^rG_K$ denotes the map induced by the inclusion $G^r_{K,\log}\subseteq G^r_K$. By Proposition \ref{diaglnl}, the image of $\mathrm{char}(\chi)\in \Hom_{\overline F}(\fm^{r}_{\overline K}/\fm^{r+}_{\overline K},\Omega^1_{\sO_{X,\xi}}\otimes_{\sO_{X,\xi}}\overline F)$ in $\Hom_{\overline F}(\fm^{r}_{\overline K}/\fm^{r+}_{\overline K},\Omega^1_F(\log)\otimes_{F}\overline F)$ induced by the composition of canonical maps
\begin{equation*}
\phi:\Omega^1_{\sO_{X,\xi}}\otimes_{\sO_{X,\xi}}F\rightarrow\Omega^1_F\rightarrow\Omega^1_F(\log)
\end{equation*}
 is $\mathrm{rsw}(\chi\circ\theta_r)$. It is non-zero since the refined Swan conductor is an injective map. Let $f$ be a uniformizer of $\sO_{X,\xi}$. The kernel $\ker(\phi)$ is the $1$-dimensional $F$-vector space generated by $df\otimes 1$. Hence, the line $L_{\chi}$ in $\bT^*X\times_X\Spec(F_{\chi})$ corresponding to $\mathrm{char}(\chi)$ is not $\bT^*_DX\times_D\Spec(F_{\chi})$, since $df\otimes 1$ is a base of the $1$-dimensional fiber $\bT^*_DX\times_D \xi\subseteq \bT^*_\xi X$. Therefore, for each center character $\chi$ appears in $M$, the cycle $[\overline L_{\chi}]$ in $CC(j_!\sF)$ contributed by $\chi$ is not $\bT^*_DX$ in a Zariski neighborhood of $\xi\in X$ \eqref{wrcbcc}. Since $SS(j_!\sF)$ is the support of $CC(j_!\sF)$, we obtain (iii).
\end{proof}

\begin{proposition}\label{propswdim=dimtot}
If we have
\begin{equation}\label{swdim=dimtot}
\sw_D(j_!\sF)+\rk_{\Lambda}\sF=\dimtot_D(j_!\sF),
\end{equation}
then, after replacing $X$ by an open neighborhood of $\xi$ such that $D$ is smooth over $\Spec(k)$ and that the ramification of $\sF$ along $D$ is non-degenerate (\ref{wrccss}), we have
\begin{equation}\label{sslikecurve}
SS(j_!\sF)=\bT^*_XX\bigcup\bT^*_DX.
\end{equation}

When $\sF$ has rank $1$, the converse of the above statement is also true.
\end{proposition}
\begin{proof}
We denote by $M$ the finitely generated $\Lambda$-module with a continuous $G_K$-action corresponding to $\sF|_{\eta}$. The equality \eqref{swdim=dimtot} implies that, for each $r\in \mathbb Q_{\geq 1}$, we have $M^{(r)}=M^{(r-1)}_{\log}$. By enlarging $\Lambda$, we may assume that it contains $p$-th roots of unity. Let $M^{(r)}=\bigoplus_{\chi\in X(r)}M^{(r)}_{\chi}$ be the center character decomposition \eqref{ccdecomp}. Fix an rational number $r>1$ such that $M^{(r)}\neq 0$.  Since $M^{(r)}=M^{(r-1)}_{\log}$, for each center character $\chi:\Gr^r G_K\rightarrow \Lambda^{\times}$ appears in $M^{(r)}$, the composition $\chi\circ\theta_r:\Gr^r_{\log} G_K\rightarrow \Lambda^{\times}$ is a trivial character, where $\theta_r:\Gr^r_{\log} G_K\rightarrow \Gr^r G_K$ denotes the canonical map induced by the inclusion $G^r_{K,\log}\subseteq G^r_K$. By Propostion \ref{diaglnl}, the image of $\mathrm{char}(\chi)\in \Hom_{\overline F}(\fm_{\overline K}^r/\fm_{\overline K}^{r+},\Omega^1_{\sO_{X,\xi}}\otimes_{\sO_{X,\xi}}\overline F)$ in $\Hom_{\overline F}(\fm_{\overline K}^r/\fm_{\overline K}^{r+},\Omega^1_{F}(\log)\otimes_F\overline F)$ induced by the composition of canonical maps
\begin{equation*}
\phi:\Omega^1_{\sO_{X,\xi}}\otimes_{\sO_{X,\xi}}F\rightarrow\Omega^1_F\rightarrow\Omega^1_F(\log)
\end{equation*}
is $\mathrm{rsw}(\chi\circ\theta_r)=0$. Let $f$ be a uniformizer of $\sO_{X,\xi}$. The kernel $\ker(\phi)$ is the $1$-dimensional $F$-vector space generated by $df\otimes 1$. Hence,
the line in $\bT^*X\times_X\Spec(F_\chi)$ corresponding to $\mathrm{char}(\chi)$ is $\bT^*_DX\times_D\Spec(F_\chi)$, since $df\otimes 1$ is a base of the $1$-dimensional fiber $\bT^*_DX\times_D \xi\subseteq \bT^*_\xi X$. Therefore, the cycle $[\overline L_{\chi}]$ in $CC(j_!\sF)$ contributed by $\chi$ is $\bT^*_DX$ in a Zariski neighborhood of $\xi\in X$ \eqref{wrcbcc}. In summary, the cycle in $CC(j_!\sF)$ contributed by each center character of $M$ supported in $\bT^*_DX$. On the other hand, the cycle in $CC(j_!\sF)$ contributed by the tame part $M^{(1)}$ is also supported on $\bT^*_DX$. Hence, after replacing $X$ by a Zariski neighborhood of $\xi\in X$ such that $D$ is smooth over $\Spec(k)$ and the ramification of $\sF$ along $D$ is non-degenerate, we have
$CC(j_!\sF)=(-1)^{\dim_kX}([\bT^*_XX]+m\cdot[\bT^*_DX])$ for some positive integer $m$. Since $SS(j_!\sF)$ is the support of $CC(j_!\sF)$, we obtain \eqref{sslikecurve}.

When $\sF$ has rank $1$, we have either (Theorem \ref{xlmain} and \eqref{swdimtotcomp})
\begin{equation*}
\sw_D(j_!\sF)+1=\dimtot_D(j_!\sF)\ \ \ \textrm{or}\ \ \ \sw_D(j_!\sF)=\dimtot_D(j_!\sF).
\end{equation*}
Hence, the converse of the statement in Proposition \ref{propswdim=dimtot} is deduced by Proposition \ref{sw=dimtot}.
\end{proof}

\begin{remark}
When $\sF$ has rank $1$, Proposition \ref{sw=dimtot} and Proposition \ref{propswdim=dimtot} are known in \cite{ya}.
\end{remark}

\section{Semi-continuity of Swan divisors}

\subsection{}
In this section, we assume that $k$ is perfect. Let $S$ denotes an irreducible $k$-scheme, $\eta$ its generic point, $g:X\rightarrow S$ a smooth morphism of $k$-schemes, $D$ an effective Cartier divisor on $X$ relative to $S$, $U$ the complement of $D$ in $X$ and $j:U\rightarrow X$ the canonical injection. Let $\{D_i\}_{i\in I}$ $(I=\{1,2,\cdots,m\})$ be the set of irreducible components of $D$. We assume that, for each $i\in I$, the restriction $g|_{D_i}:D_i\rightarrow S$ is smooth. For any point $s$ of $S$, we denote by $\iota_s:X_s\rightarrow X$ the canonical injection. For any geometric point $\bar s\rightarrow S$, we denote by $\iota_{\bar s}:X_{\bar s}\rightarrow X$ the base change of $\bar s\rightarrow S$ by $g:X\rightarrow S$.

Let $\sF$ be a locally constant and constructible sheaf of $\Lambda$-modules on $U$. We define the {\it generic total dimension divisor of} $j_!\sF$ on $X$ and denote by $\GDT_g(j_!\sF)$ the unique Cartier divisor of $X$ supported on $D$ such that $\iota^*_\eta(\GDT_g(j_!\sF))=\DT_{X_{\eta}}(j_!\sF|_{X_{\eta}})$ (\ref{dtdiv}). We define the {\it generic logarithmic total dimension divisor of} $j_!\sF$ on $X$ and denote by $\GLDT_g(j_!\sF)$ the unique Cartier divisor of $X$ supported on $D$ such that $\iota^*_\eta(\GLDT_g(j_!\sF))=\LDT_{X_{\eta}}(j_!\sF|_{X_{\eta}})$ (\ref{ldtdiv}).

\begin{theorem}[{\cite[Theorem 4.3]{HY}}]\label{nonlogsemicont}
There exists a Zariski open dense subscheme $V$ of $S$ such that,
\begin{itemize}
\item[(i)]
for any $s\in V$, we have $\iota_s^*(\GDT_g(j_!\sF))=\DT_{X_s}(j_!\sF|_{X_s})$;
\item[(ii)]
for any $s\in S-V$, we have $\iota_s^*(\GDT_g(j_!\sF))\geq\DT_{X_s}(j_!\sF|_{X_s})$.
\end{itemize}
\end{theorem}

In the following of this section, we prove Theorem \ref{logsemicontintro}, a logarithmic ramification version of Theorem \ref{nonlogsemicont}. Essentially, it sufficient to treat case where $S$ is irreducible as follows (Theorem \ref{semithm}). When $g:X\rightarrow S$ is of relative dimension $1$, both of the two theorems are equivalent to Deligne and Laumon's semi-continuity property for Swan conductors \cite[2.1.1]{lau}.

\begin{theorem}\label{semithm}
There exists a Zariski open dense subscheme $V$ of $S$ such that,
\begin{itemize}
\item[(i)]
for any $s\in V$, we have $\iota_s^*(\GLDT_g(j_!\sF))=\LDT_{X_s}(j_!\sF|_{X_s})$;
\item[(ii)]
for any $s\in S-V$, we have $\iota_s^*(\GLDT_g(j_!\sF))\geq\LDT_{X_s}(j_!\sF|_{X_s})$.
\end{itemize}
\end{theorem}

\subsection{}
To prove Theorem \ref{semithm}, we may assume that $S$ is integral. By \cite[Theorem 4.1]{deJ}, we have a connected smooth $k$-scheme $S'$ and a generically finite and surjective morphism $\pi:S'\rightarrow S$. We put $X'=X\times_SS'$ and let $g':X'\rightarrow S'$ be the base change of $g:X\rightarrow S$ and $\pi':X'\rightarrow X$ the base change of $\pi:S'\rightarrow S$. By \ref{swbc}, we have
\begin{equation*}
\pi'^*(\GLDT_g(j_!\sF))=\GLDT_{g'}(\pi'^*(j_!\sF)).
\end{equation*}
Let $s'$ be a point of $S'$, $s$ its image in $S$ and $\pi'_{s'}:X'_{s'}\rightarrow X_s$ the canonical projection.
By \ref{swbc} again, we have
\begin{equation*}
\pi'^*_{s'}(\LDT_{X_s}(j_!\sF|_{X_s}))=\LDT_{X'_{s'}}(j_!\sF|_{X'_{s'}}).
\end{equation*}
Since $\pi:S'\rightarrow S$ is generically finite, for any open dense subset $U'$ of $S'$, the image $\pi(S')$ contains a open dense subset of $S$. Notice that a Cartier divisor on $X_s$ supported on $D_s$ is effective if and only if its pull-back to $X'_{s'}$ is effective. Hence, to prove Theorem \ref{semithm}, we may replace $S$ by $S'$, i.e., we may assume that $S$ is connected and smooth over $\Spec(k)$.  Let $\bar k$ be an algebraic closure of $k$, $\overline S$ an irreducible component of $S_{\bar k}=S\times_k\bar k$. By the same argument as above, we may further replacing $S$ by $\overline S$. In summary, to prove Theorem \ref{semithm}, we may assume that $k$ is algebraically closed and $S$ that is a connected and smooth $k$-scheme.

\subsection{}
In the following of this section, we assume that $k$ is algebraically closed and that $S$ is a connected and smooth $k$-scheme.

\begin{proposition}[{Proof of (ii) of Theorem \ref{semithm}}]\label{semithmpart2}
There exists a Zariski open dense subscheme $V$ of $X$ such that, for each point $s$ of $V$, we have $\iota_s^*(\GLDT_g(j_!\sF))=\LDT_{X_s}(j_!\sF|_{X_s})$.
\end{proposition}
\begin{proof}
For each $i\in I$, we denote by $D_i^{\circ}$ a Zariski open dense subscheme of $D_i$ such that $D^\circ_i\bigcap D_{i'}=\emptyset$ for $i'\in I\backslash\{i\}$. By \cite[III, 9.6.1(ii)]{EGA4}, there exists an open dense subset $W$ of $S$ such that, for any $s\in W$, the injection $\coprod_{i\in I}(D^\circ_i)_s\rightarrow D_s$ is dominant. Hence, for any $s\in W$ and any $i,i'\in I$ $(i\neq i')$, $(D_i)_s$ and $(D_{i'})_s$ do not have same irreducible components. Therefore, to prove the proposition, we may assume that $D$ is irreducible.

This is a Zariski local problem in a neighborhood of the generic point of $D$. We may assume that $S$ and $X$ are affine and $D$ is defined by an element $f$ of $\Gamma(X,\sO_X)$. Let $\beta$ be a positive integer co-prime to $p$ and greater than $2\cdot\rk_{\Lambda}\sF+1$. Let
\begin{equation*}
X_{\beta}=\Spec(\sO_X[T]/(T^\beta-f))
\end{equation*}
be a cyclic cover of $X$ of degree $\beta$ tamely ramified along $D$,  $h_\beta:X_{\beta}\rightarrow X$ the canonical projection, $g_{\beta}:X_{\beta}\rightarrow S$ the composition of $h_{\beta}$ and $g$ and $D_{\beta}$ the smooth divisor $(T)=(D\times_XX_{\beta})_{\mathrm{red}}$ of $X_{\beta}$ relative to $S$. Notice that $g_{\beta}:X_{\beta}\rightarrow S$ and $g_{\beta}|_{D_{\beta}}:D_{\beta}\rightarrow S$ are smooth. For any point $s$ of $S$, we denote by $\iota_{\beta,s}:(X_{\beta})_s\rightarrow X_{\beta}$ the canonical injection and by $h_{\beta,s}:(X_{\beta})_s\rightarrow X_s$ the fiber of $h_{\beta}:X_{\beta}\rightarrow X$ at $s$.
Applying \cite[Theorem 4.3]{HY} (cf. Theorem \ref{nonlogsemicont}) to the sheaf $h^*_{\beta}(j_!\sF)$ on $X_{\beta}$ and the projection $g_{\beta}:X_{\beta}\rightarrow S$, we can find an Zariski open dense subset $V_{\beta}$ of $S$ such that, for any $s\in V_{\beta}$, we have
\begin{equation*}
\iota_{\beta,s}^*(\GDT_{g_\beta}(j_!\sF))=\DT_{(X_{\beta})_s}(j_!\sF|_{(X_{\beta})_s}).
\end{equation*}
 It is equivalent to that, for any $s\in V_{\beta}$,
\begin{equation}\label{Dirrform1}
\DT_{(X_{\beta})_s}(j_!\sF|_{(X_{\beta})_s})=\dimtot_{(D_{\beta})_\eta}(j_!\sF|_{(X_{\beta})_\eta})\cdot (D_{\beta})_s.
\end{equation}
By \eqref{swdimtotcomp} and Proposition \ref{unramtame}, for any $s\in V_{\beta}$, we have
\begin{align}\label{Dirrform2}
(\beta\cdot\sw_{D_\eta}(j_!\sF|_{X_{\eta}})+\rk_{\Lambda}\sF)\cdot (D_{\beta})_s& \geq\dimtot_{(D_{\beta})_{\eta}}(j_!\sF|_{(X_{\beta})_{\eta}})\cdot (D_{\beta})_s \\
&\geq \beta\cdot\sw_{D_\eta}(j_!\sF|_{X_{\eta}})\cdot (D_{\beta})_s,\nonumber
\end{align}
and
\begin{align}\label{Dirrform3}
h_{\beta,s}^*(\SW_{X_s}(j_!\sF|_{X_s}))+\rk_{\Lambda}\sF\cdot (D_\beta)_s&\geq\DT_{(X_{\beta})_s}(j_!\sF|_{(X_{\beta})_s})\\
&\geq h_{\beta,s}^*(\SW_{X_s}(j_!\sF|_{X_s})).\nonumber
\end{align}
By \eqref{Dirrform1}, \eqref{Dirrform2} and \eqref{Dirrform3}, for any $s\in V_{\beta}$, we have
\begin{align*}
(\beta\cdot\sw_{D_\eta}(j_!\sF|_{X_{\eta}})+\rk_{\Lambda}\sF)\cdot(D_{\beta})_s&\geq  h_{\beta,s}^*(\SW_{D_s}(j_!\sF|_{X_s})),\\
&\geq (\beta\cdot\sw_{D_\eta}(j_!\sF|_{X_{\eta}})-\rk_{\Lambda}\sF)\cdot (D_\beta)_s,
\end{align*}
i.e.,
\begin{align*}
\left(\sw_{D_\eta}(j_!\sF|_{X_{\eta}})+\frac{1}{\beta}\cdot\rk_{\Lambda}\sF\right)\cdot D_s&\geq  \SW_{D_s}(j_!\sF|_{X_s}),\\
&\geq \left(\sw_{D_\eta}(j_!\sF|_{X_{\eta}})-\frac{1}{\beta}\cdot\rk_{\Lambda}\sF\right)\cdot D_s,
\end{align*}
Observe that $\beta>2\cdot\rk_\Lambda\sF+1$ and that $\sw_{D_\eta}(j_!\sF|_{X_\eta})$ and coefficients of $\SW_{D_s}(j_!\sF|_{X_s})$ are non-negative integers (Theorem \ref{xlmain}). For any $s\in V_{\beta}$, we must have
\begin{equation}\label{sws=eta}
\sw_{D_\eta}(j_!\sF|_{X_{\eta}})\cdot D_s=\SW_{D_s}(j_!\sF|_{X_s}).
\end{equation}
Add $\rk_\Lambda\sF\cdot D_s$ on both sides of \eqref{sws=eta}, we obtain that, for any $s\in V_\beta$,
\begin{equation*}
  \iota_s^*(\GLDT_g(j_!\sF))=\LDT_{X_s}(j_!\sF|_{X_s}).
  \end{equation*}

\end{proof}

\begin{proposition}\label{part1closedpoint}
For any closed point $s$ of $S$, we have
\begin{equation}\label{clptLDT}
\iota_s^*(\GLDT_g(j_!\sF))\geq \LDT_{X_s}(j_!\sF|_{X_s}).
\end{equation}
\end{proposition}

\begin{proof}
We fix a closed point $s$ of $S$. This is a local problem for the Zariski topology of $X$. We may assume that $S$ and $X$ are integral and affine such that $D_s\neq \emptyset$. We put $n=\dim_k X-\dim_kS$. When $n=1$, the map $f:X\rightarrow S$ is a relative curve, the inequality \eqref{clptLDT} is due to \cite[2.1.1]{lau}. We focus on the case where $n\geq 2$.  We may assume that $D_s$ is irreducible and that $(D_s)_{\mathrm{red}}=(D_i)_s$ for each $i\in I$.
Let $\beta$ be an integer co-prime to $p$, $f_1$ an element of $\Gamma(X,\sO_X)$ that defines $D_1$ and
\begin{equation*}
X'=\Spec(\sO_X[T]/(X^\beta-f_1))
\end{equation*}
a cyclic cover of $X$ of degree $\beta$ tamely ramified along $D_1$, $h':X'\rightarrow X$ the canonical projection, $g':X'\rightarrow S$ the composition of $h'$ and $g$, $D'$ the reduced Cartier divisor $(D\times_XX')_{\mathrm{red}}$ on $X$, $D'_1$ the smooth divisor $(T)=(D_1\times_XX')_{\mathrm{red}}$ on $X'$ and $D'_i$ the reduced Cartier divisor $D_i\times_XX'$ on $X'$ for each $i\in I\backslash \{1\}$. Observe that $D'=\sum_{i\in I}D'_i$ and that $\beta\cdot(D'_1)_s=(D'_i)_s$ for each $i\in I\backslash\{1\}$.

Applying Proposition \ref{findcurve} to the sheaf $h'^*(j_!\sF)|_{X'_s}$ on $X'_s$, we can find a closed point $x'$ of $(D'_1)_s$ with image $x$ in $(D_1)_s$, a smooth $k$-curve $C$ and an immersion $\gamma:C\rightarrow X'_s$ such that
\begin{itemize}
\item[(i)] $C\bigcap (D'_1)_s=x'$ and $m_{x'}(\gamma^*(D'_1)_s)=1$;
\item[(ii)] $\gamma:C\rightarrow X'_s$ is $SS(h'^*(j_!\sF)|_{X'_s})$-transversal at $x'$;
\item[(iii)] the composition $h'_s\circ \gamma:C\rightarrow X_s$ is also an immersion, where $h'_s:X'_s\rightarrow X_s$ denotes the fiber of $h':X'\rightarrow X$.
\end{itemize}
By Proposition \ref{conj2strong}, we see that
\begin{equation}\label{useth6.4}
\dimtot_{x'}(j_!\sF|_C)\geq \beta\cdot\sw_{(D_1)_s}(j_!\sF|_{X_s}).
\end{equation}

Choose a regular system of parameters $\bar t_1,\cdots, \bar t_n$ of $\sO_{X'_s,x'}$ such that the ideal $(\bar t_1,\cdots,\bar t_{n-1})$ defines $\sO_{C,x'}$ and that $(\bar t_n)$ defines $\sO_{(D'_1)_s,x'}$.
For any $1\leq i\leq n-1$, choose a lifting $t_i\in \sO_{X',x'}$ of $\bar t_i\in \sO_{X'_s,x'}$.
After replacing $X$ by an open neighborhood of $x$, the $\sO_{S,s}$-homomorphism
\begin{equation*}
\sO_{S,s}[T_1,\cdots, T_{n-1}]\rightarrow\sO_{X',x'}, \ \ \ T_i\mapsto t_i,
\end{equation*}
induces an $S$-morphism $r':X'\rightarrow \bA^{n-1}_S$. It satisfies following conditions after shrinking $X$ again
\begin{itemize}
\item[(1)] it is smooth and of relative dimension $1$;
\item[(2)] $r'|_{D'}:D'\rightarrow \mathbb A^{n-1}_S$ is quasi-finite and flat, for each $i\in I\backslash \{1\}$, the restriction $r'|_{D'_i}:D'_i\rightarrow \mathbb A^{n-1}_S$ is quasi-finite and flat and  $r'|_{D'_1}:D'_1\rightarrow \mathbb A^{n-1}_S$ is \'etale (\cite[I, chapitre 0,15.1.16]{EGA4});
\item[(3)] the curve $C$ is the pre-image $r'^{-1}(s\times o)$, where $s\times o$ denotes the product of $s\in S$ and the origin $o\in\bA^{n-1}_k$.
\end{itemize}
We denote by $Z$ the scheme theoretic image of $\bigcup_{i,i'\in I, i\neq i'}(D'_i\bigcap D'_{i'})\subseteq D'$ in $\bA^{n-1}_S$. By (2), it is of codimension $1$ in $\bA^{n-1}_S$. Choose a section $\sigma:S\rightarrow \bA^{n-1}_S$ of the canonical projection $\pi:\bA^{n-1}_S\rightarrow S$ such that $\sigma(s)=s\times o$ and that $\sigma(S)\not\subseteq Z$. We denote by $q: Y\rightarrow S$ the base change of $r':X'\rightarrow\bA^{n-1}_S$ by $\sigma:S\rightarrow \bA^{n-1}_S$. We put $E=Y\times_{X'}D'$ and, for each $i\in I$, we put $E_i=Y\times_{X'}D_i'$. By \cite[IV,18.12.1]{EGA4}, we can find a connected \'etale neighborhood $\theta:V\rightarrow S$ of $s\in S$  such that,
\begin{itemize}
\item[(a)] the pre-image $v=\theta^{-1}(s)$ is a point;
\item[(b)] the fiber product $E_V=E\times_YV$ is a disjoint union of two schemes $Q$ and $P$;
\item[(c)] $z=(Q_v)_{\mathrm{red}}$ is a point and it is the unique pre-image of $x'\in E_s\subseteq Y_s\subseteq X'$.
\item[(d)] $Q$ is finite and flat over $V$, for any $i\in I\backslash\{1\}$, $Q_i=E_i\times_EQ$ is finite and flat over $V$ and $Q_1=E_1\times_EQ$ is isomorphic to $V$.
\end{itemize}
We have the following commutative diagram
\begin{equation}\label{bigdiag}
\xymatrix{\relax
Q\ar@{}|-(0.5){\Box}[rd]\ar[r]\ar[d]& E_V\ar[d]\ar[r]\ar@{}|-(0.5){\Box}[rd]& E\ar[d]\ar[r]\ar@{}|-(0.5){\Box}[rd]& D'\ar[d]\ar[r]&D\ar[d]\\
    Y^0_{V} \ar[r]    &Y_V \ar[d]_-(0.5){q_V}\ar[r]^-(0.5){\theta_V}\ar@{}|-(0.5){\Box}[rd]&Y \ar[d]^q\ar[r]^-(0.5){\sigma'}\ar@{}|-(0.5){\Box}[rd]& X'\ar[d]^-(0.5){r'}\ar[r]^-(0.5){h'}& X\ar[d]^-(0.5){g}\\
          &V \ar[r]_-(0.5){\theta}& S\ar[r]\ar[r]_-(0.5){\sigma}&\bA^{n-1}_S\ar[r]\ar[r]_-(0.5){\pi}& S}
\end{equation}
where $Y^0_V$ denotes the complement of $P$ in $Y_V$.
Let $\eta$ be the generic point of $S$ and $\bar\eta$ an algebraic geometric point above $\eta$ that factors through $V$. Recall that $X_{\bar\eta}=X\times_S\bar\eta$, that $D_{\bar\eta}=D\times_S\bar\eta$, that $(D_i)_{\bar\eta}=D_i\times_S\bar\eta$, that $X'_{\bar\eta}=X'\times_S\bar\eta$, that $D'_{\bar\eta}=D'\times_S\bar\eta$ and that $(D'_i)_{\bar\eta}=D'_i\times_S\bar\eta$. We put $Y_{\bar\eta}=Y\times_S\bar\eta$, $Y^0_{\bar\eta}=Y^0_V\times_V\bar\eta$, $E_{\bar\eta}=E\times_S\bar\eta$, $(E_i)_{\bar\eta}=E_i\times_S\bar\eta$,  $Q_{\bar\eta}=Q\times_S\bar\eta$ and $(Q_i)_{\bar\eta}=Q_i\times_S\bar\eta$. Since $\sigma(\eta)\not\in Z\subseteq \bA^{n-1}_S$, for any $i, i'\in I$ ($i\neq i'$), we have
$(E_i)_{\bar\eta}\bigcap (E_{i'})_{\bar\eta}=\emptyset$, and hence, $(Q_i)_{\bar\eta}\bigcap (Q_{i'})_{\bar\eta}=\emptyset$.
Applying Deligne and Laumon's semi-continuity property of Swan conductors \cite[2.1.1]{lau} to the sheaf $j_!\sF|_{Y_V^0}$ on the relative curve $q_V:Y_V^0\rightarrow V$, we get
\begin{equation}\label{dellauuse}
\sum_{i\in I}\sum_{y\in(Q_i)_{\bar\eta}}\dimtot_y(j_!\sF|_{Y^0_{\bar\eta}})\geq \dimtot_{x'}(j_!\sF|_C).
\end{equation}
Take the geometric generic fiber of \eqref{bigdiag}, we have the following commutative diagram
\begin{equation*}
\xymatrix{\relax
Q_{\bar\eta}\ar@{}[rd]|-(0.5){\Box}\ar[d]\ar[r]&E_{\bar\eta}\ar@{}[rd]|-(0.5){\Box}\ar[r]\ar[d]&D'_{\bar\eta}\ar[r]\ar[d]&D_{\bar\eta}\ar[d]\\
Y^0_{\bar\eta}\ar[r]&Y_{\bar\eta}\ar@{}[rd]|-(0.5){\Box}\ar[r]^-(0.5){\sigma'_{\bar\eta}}\ar[d]_-(0.5){q_{\bar\eta}}&X'_{\bar\eta}\ar[r]^-(0.5){h'_{\bar\eta}}\ar[d]^-(0.5){r'_{\bar\eta}}&X_{\bar\eta}\ar[d]^-(0.5){g_{\bar\eta}}\\
&\bar\eta\ar[r]_-(0.5){\sigma_{\bar\eta}}&\bA^{n-1}_{\bar\eta}\ar[r]_-(0.5){\pi_{\bar\eta}}&\bar\eta}
\end{equation*}
Notice that $ D'_{\bar\eta}=(D_{\bar\eta}\times_{\bar\eta}X'_{\bar\eta})_{\mathrm{red}}$, that $ (D'_1)_{\bar\eta}=((D_1)_{\bar\eta}\times_{\bar\eta}X'_{\bar\eta})_{\mathrm{red}}$ and that, for any $i\in I\backslash\{1\}$, $ (D'_i)_{\bar\eta}=(D_i)_{\bar\eta}\times_{\bar\eta}X'_{\bar\eta}$. Hence, $(Q_1)_{\bar\eta}=((D_1)_{\bar\eta}\times_{X'_{\bar\eta}}Y^0_{\bar\eta})_{\mathrm{red}}$ and that $(Q_i)_{\bar\eta}=(D_i)_{\bar\eta}\times_{X_{\bar\eta}}Y^0_{\bar\eta}$ for any $i\in I\backslash\{1\}$. By (d), the fiber $(Q_1)_{\bar\eta}$ is a closed point of $Y^0_{\bar\eta}$ and, for each $i\in I\backslash\{1\}$,
\begin{equation*}
\mathrm{length}_{\bar\eta}((Q_i)_{\bar\eta})=\mathrm{length}_{k}((Q_i)_v)=m_{x'}(\gamma^*(D'_i)_s)=\beta.
\end{equation*}
Applying \cite[Theorem 4.2]{HY} (cf. Theorem \ref{hymain}) to the sheaf $j_!\sF|_{X'_{\bar\eta}}$ and the injection $\sigma'_{\bar\eta}:Y^0_{\bar\eta}\rightarrow X'_{\bar\eta}$ along the divisor $(D'_1)_{\bar\eta}$, we have
\begin{equation}\label{hyuseX'}
\dimtot_{(D'_1)_{\eta}}(j_!\sF|_{X'_\eta})\geq \dimtot_{(Q_1)_{\bar\eta}}(j_!\sF|_{Y^0_{\bar\eta}}).
\end{equation}
For any $i\in I\backslash\{1\}$, applying \cite[Theorem 4.2]{HY} (cf. Theorem \ref{hymain}) to the sheaf $j_!\sF|_{X_{\bar\eta}}$ and the quasi-finite morphism $\sigma'_{\bar\eta}:Y^0_{\bar\eta}\rightarrow X_{\bar\eta}$ along the divisor $(D_i)_{\bar\eta}$, we have,

\begin{equation}\label{hyuseX}
\beta\cdot\dimtot_{(D_i)_{\eta}}(j_!\sF|_{X_\eta})\geq\sum_{y\in (Q_i)_{\bar\eta}}\dimtot_y(j_!\sF|_{Y^0_{\bar\eta}}).
\end{equation}
By \eqref{dellauuse}, \eqref{hyuseX'} and \eqref{hyuseX} , we get
\begin{equation}\label{keyineqdt}
\dimtot_{(D'_1)_{\eta}}(j_!\sF|_{X'_\eta})+\sum_{i\in I\backslash\{1\}}\beta\cdot\dimtot_{(D_i)_{\eta}}(j_!\sF|_{X_\eta})\geq\dimtot_z(j_!\sF|_C).
\end{equation}
Since $h'_\eta:X'_{\eta}\rightarrow X_{\eta}$ is tamely ramified along $(D_1)_\eta$, by \eqref{swdimtotcomp} and Proposition \ref{unramtame}, we have
\begin{equation}\label{useswdttame}
\beta\cdot\sw_{(D_1)_{\eta}}(j_!\sF|_{X_{\bar\eta}})+\rk_{\Lambda}\sF=\sw_{(D'_1)_{\eta}}(j_!\sF|_{X'_{\bar\eta}})+\rk_{\Lambda}\sF\geq \dimtot_{(D'_1)_{\eta}}(j_!\sF|_{X'_\eta}).
\end{equation}
For any $i\in I\backslash\{1\}$, we have \eqref{swdimtotcomp}
\begin{equation}\label{useswdt}
\sw_{(D_i)_{\eta}}(j_!\sF|_{X_\eta})+\rk_\Lambda\sF\geq\dimtot_{(D_i)_{\eta}}(j_!\sF|_{X_\eta}).
\end{equation}
By \eqref{useth6.4}, \eqref{keyineqdt}, \eqref{useswdttame} and \eqref{useswdt}, we have
\begin{align}
\beta\cdot\sw_{(D_1)_{\eta}}(j_!\sF|_{X_{\bar\eta}})+\sum_{i\in I\backslash\{1\}}\beta\cdot\sw_{(D_i)_{\eta}}(j_!\sF|_{X_\eta})&+((\sharp I-1)\beta+1)\cdot\rk_\Lambda\sF \label{keyineqsw}\\
&\geq\beta\cdot\sw_{(D_1)_s}(j_!\sF|_{X_s}).\nonumber
\end{align}
Divide both sides of \eqref{keyineqsw} by $\beta$, and then pass $\beta\rightarrow\infty$. We obtain
\begin{align}
\sum_{i\in I}&(\sw_{(D_i)_{\eta}}(j_!\sF|_{X_\eta})+\rk_\Lambda\sF)\geq\sw_{(D_1)_s}(j_!\sF|_{X_s})+\rk_\Lambda\sF,
\end{align}
which proves \eqref{clptLDT}.
\end{proof}

\subsection{}{\it Proof of Theorem} \ref{semithm}.
Since (ii) of Theorem \ref{semithm} is proved in Proposition \ref{semithmpart2}, we can find a Zariski open dense subset $V$ of $S$ such that, for any $s\in V$,
\begin{equation*}
\iota_s^*(\GLDT_g(j_!\sF))=\LDT_{X_s}(j_!\sF|_{X_s}).
\end{equation*}
Let $t$ be a point of $S-V$. If $t$ is a closed point, we have (Proposition \ref{part1closedpoint})
\begin{equation*}
\iota_t^*(\GLDT_g(j_!\sF))\geq\LDT_{X_t}(j_!\sF|_{X_t}).
\end{equation*}
If $t$ is not a closed point of $S-V$, let  $T$ be the smooth part of $\overline{\{t\}}$, which is an open dense subset of $\overline{\{t\}}$.
We have the following commutative diagram
\begin{equation*}
\xymatrix{\relax
X_s\ar[r]^-(0.5){\rho_s}\ar[rd]_{\iota_s}&X_T\ar@{}[rd]|-(0.5){\Box}\ar[d]_-(0.5){\iota_T}\ar[r]^-(0.5){g_T}&T\ar[d]\\
&X\ar[r]_-(0.5)g&S}
\end{equation*}
where $s$ denotes a point of $T$ and $\rho_s:X_s\rightarrow X_T$ is the canonical injection. Applying Proposition \ref{semithmpart2} to  $j_!\sF|_{X_T}$ and $g_T:X_T\rightarrow T$, we can find a Zariski open dense subset $W\subseteq T$ such that, for any closed point $s$ of $W$, we have
\begin{equation}\label{apply8.6}
\rho^*_s(\GLDT_{g_T}(j_!\sF|_{X_T}))=\LDT_{X_s}(j_!\sF|_{X_s}).
\end{equation}
By Proposition \ref{part1closedpoint}, we have
\begin{equation}\label{apply8.5}
\rho_s^*(\iota^*_T(\GLDT_g(j_!\sF)))\geq\LDT_{X_s}(j_!\sF|_{X_s}).
\end{equation}
By \eqref{apply8.6} and \eqref{apply8.5}, we obtain that, for any closed point $s\in W$,
\begin{equation*}
\rho_s^*(\iota^*_T(\GLDT_g(j_!\sF)))\geq \rho_s^*(\GLDT_{g_T}(j_!\sF|_{X_T})).
\end{equation*}
It implies that
\begin{equation}\label{ineqoverT}
\iota^*_T(\GLDT_g(j_!\sF))\geq \GLDT_{g_T}(j_!\sF|_{X_T}).
\end{equation}
Applying $\rho_t^*$ on both sides of \eqref{ineqoverT}, we obtain
\begin{equation*}
\iota_t^*(\GLDT_g(j_!\sF))\geq\LDT_{X_t}(j_!\sF|_{X_t}).
\end{equation*}
Hence, (i) of Theorem \ref{semithm} is proved. \hfill$\Box$

\end{document}